\documentclass[a4paper,11pt]{amsart}
\usepackage{amsmath,amssymb,amsthm,amsfonts}
\usepackage{color}
\usepackage[all]{xy}
\usepackage{tikz-cd}
\usepackage{cancel}
\usepackage{hyperref}
\newcommand{\EE}{\mathcal{E}}
\newcommand{\LL}{\mathcal{L}}
\newcommand{\FF}{\mathcal{F}}
\newcommand{\GG}{\mathcal{G}}
\newcommand{\HH}{\mathcal{H}}
\newcommand{\kk}{\mathbb{K}}

\newcommand{\ZZ}{\mathbb{Z}}

\newcommand{\CC}{\mathcal{C}}
 
\newcommand{\OO}{\mathcal{O}}

\newcommand{\VV}{\mathcal{V}}

\DeclareMathOperator{\coker}{coker}
\DeclareMathOperator{\im}{Im}
\DeclareMathOperator{\Hom}{Hom}
\DeclareMathOperator{\Art}{\bf{Art}}
\DeclareMathOperator{\Grpds}{\bf{Grpds}}
\DeclareMathOperator{\Set}{\bf{Sets}}
\DeclareMathOperator{\Spec}{Spec}
\DeclareMathOperator{\Id}{Id}

\DeclareMathOperator{\MC}{MC}
\DeclareMathOperator{\Def}{Def}
\DeclareMathOperator{\Del}{Del}

\DeclareMathOperator{\End}{End}
\DeclareMathOperator{\Mor}{Mor}
\DeclareMathOperator{\Ext}{Ext}

\DeclareMathOperator{\Tot}{Tot}

\newcommand{\h}{\mathfrak{h}}
\newcommand{\g}{\mathfrak{g}}
\newcommand{\m}{\mathfrak{m}}

\newtheorem*{theorems}{Theorem}
\newtheorem{theorem}{Theorem}[section]
\newtheorem{corollary}[theorem]{Corollary}

\newtheorem{proposition}[theorem]{Proposition}
\newtheorem{lemma}[theorem]{Lemma}
\theoremstyle{definition}
\newtheorem{definition}[theorem]{Definition}
\newtheorem{remark}[theorem]{Remark}
\newtheorem{example}[theorem]{Example}

\definecolor{rosso}{RGB}{162,0,0}
\definecolor{verde}{RGB}{0,100,0}
\definecolor{blu}{RGB}{0,0,162}

\definecolor{ballblue}{rgb}{0.0, 0.5, 1.0}

\title{Deformations of morphisms of coherent sheaves}
\author{Donatella Iacono}
\address{\newline Dipartimento di Matematica,
	\newline  Universit\`a degli Studi di Bari Aldo Moro,
	\hfill\newline Via E. Orabona 4,
	70125 Bari, Italy.}
\email{donatella.iacono@uniba.it}

\author{Emma Lepri}
\address{
	\newline
	Dipartimento di Matematica ``Giuseppe Peano'',
	\hfill\newline
	Universit\`a degli Studi di Torino,
	\hfill\newline
	via Carlo Alberto 10,
	10123 Torino, Italy}
\email{emma.lepri@unito.it; elena.martinengo@unito.it}

\author{Elena Martinengo}


\begin{document}
	\maketitle 
	
	\begin{abstract}
We generalise Hinich's Theorem of descent of Deligne groupoids to the case where the dgLas involved have no negative cohomology. We apply this result to study the infinitesimal deformations of a morphism $\alpha: {\mathcal F} \to {\mathcal G}$ of coherent sheaves,  where both the sheaves $ {\mathcal F}$ and $ {\mathcal G}$ and the map $\alpha$ can be deformed, on  a smooth variety over a field of characteristic zero. In particular, we provide an explicit dgLa  that controls these deformations via the Deligne functor, applying  the Thom--Whitney  totalisation to a specific semicosimplicial dgLa, constructed from geometrical data.

	\end{abstract}

	\tableofcontents
	\addtocontents{toc}{\protect\setcounter{tocdepth}{1}}

In  \cite{hinich}, Hinich proved  a  theorem on descent of Deligne groupoids   that connects the classical Kodaira--Spencer approach to deformation theory, in which a global deformation is obtained by gluing the local ones, to the more recent approach via differential graded Lie algebras (dgLas), in which a deformation problem is controlled by a dgLa via the Deligne functor. 
The strength of Hinich's Theorem is that it allows in concrete examples to find a specific and quite explicit dgLa starting from a geometrical analysis.

In particular, the classical approach to a geometric deformation problem naturally leads to a sheaf of dgLas, which yields a semicosimplicial  dgLa $\g^{\Delta}$: its \v{C}ech complex. Then the deformations are described by a totalisation process of Deligne functors: $  \Tot(\Del_{\g^{\Delta }})$. Instead, the dgLa approach encodes deformations via the Deligne functor associated to the Thom--Whitney totalisation  $\Tot (\g^{\Delta })$. Then, Hinich's Theorem provides an equivalence between the two functors in groupoids  $ \Del_{\Tot (\g^{\Delta }) } $ and $ \Tot(\Del_{\g^{\Delta }})$, whenever all the involved dgLas vanish in negative degree (see Section~\ref{sezione Descent of Deligne groupoids} for more details).

In \cite{FMM, FIM}, the equivalence between the two approaches is considered in the setting of $\Set$-valued deformation functors. In particular, in \cite{FMM}, the authors investigate the infinitesimal deformations of a locally free sheaf over any field of characteristic zero and, as a tool, they explicitly show the equivalence under the hypothesis that the dgLas are zero in negative degrees, that is, in the same hypothesis of Hinich's Theorem. In \cite{ FIM}, the authors analyse the infinitesimal deformations of a coherent sheaf. Here, since the involved sheaves of dgLas are non-zero in negative degree, they need   to explicitly  exhibit   the above equivalence of functors in sets under the milder hypothesis that the dgLas have no negative cohomology.

One of the main goals of this paper is to generalise Hinich's Theorem on descent of Deligne groupoids to the case of dgLas with negative degree but no negative cohomology.
\begin{theorems}[Theorem~\ref{generalized hinich theorem}]
	Let $\g^\Delta$ be a semicosimplicial dgLa such that $H^j(\g_i)=0$ for all $i\geq 0$ and $j<0$.
	Then, there is an equivalence of $\Grpds$-valued functors
	\[ \Phi: \Del_{\Tot (\g^{\Delta }) }   \to   \Tot(\Del_{\g^{\Delta }}).\]
\end{theorems}

In this context, 
many other authors contributed in recent years to the analysis of these functors and equivalences among them from different points of view. In \cite{getzler}, the Deligne groupoid was extended to $L_\infty$ algebras, and a theorem of descent for the Deligne groupoid of a semicosimplicial $L_\infty$ algebra concentrated in non-negative degrees is due to \cite{bandiera-tesi, bandiera}. This last result implies that descent of Deligne groupoids holds for any semicosimplicial dgLa cohomologically concentrated in non-negative degrees which is quasi-isomorphic to a semicosimplicial non-negatively graded $L_\infty$ algebra; for instance, for a semicosimplicial dgLa cohomologically concentrated in non-negative degrees equipped with a semicosimplicial contraction to its piecewise cohomology. 

The result can fail when the semicosimplicial dgLas involved have non-trivial negative cohomology (see Remark~\ref{rem.hyp-ott}); in the general case the right object to consider would be some higher Deligne groupoid \cite{bandiera-tesi, bandiera}.

Here, we would like to explicitly prove  the previous result without leaving the realm of dgLas, to keep the homotopical constructions as minimal as possible.

\medskip

Our interest in the generalisation of Hinich's Theorem actually arose  from the aim of analysing infinitesimal deformations of morphisms of coherent sheaves. 
Let $X$ be a smooth variety over a field of characteristic zero. Let $\FF$ and $\GG$ be coherent sheaves of $\OO_X$-modules on $X$ and $\alpha: \FF \to \GG$   a morphism of sheaves. 
We are interested in the infinitesimal deformations of $\alpha: \FF \to \GG$,  where both the sheaves $\FF$ and $\GG$ and the map $\alpha$ can be deformed.

We tackle this problem using similar techniques as in \cite{DE Brill-Noether, DE maps}, in which the sheaves $\FF$ and $\GG$ are locally free. Since now the sheaves are coherent, we are led to handle resolutions of locally free sheaves instead of single sheaves,  which makes the study much more involved. 

First we note that this deformation problem is equivalent to deforming the graph of $\alpha$ as a subsheaf of the direct sum $\FF \oplus \GG$, in such a way that the deformation of  $\FF \oplus \GG$ is given by a deformation of $\FF$ and a deformation of $\GG$. This approach is similar to the one developed in \cite{Dona.Tesi, Dona.def maps, DE maps}.
Since we work over a field of characteristic zero, all these deformation problems, deformations of a sheaf and of a subsheaf inside a sheaf, are studied via the  semicosimplicial dgLas techniques as developed in \cite{FMM, FIM}.
 In particular, we are able to associate to all of the deformation problems involved a \v{C}ech complex, which is  a semicosimplicial dgLa. 
Then, the geometric deformation data are naturally encoded in the totalisation of the semicosimplicial Deligne functor of the corresponding \v{C}ech complex. 

In this perspective, the set-valued deformations of a coherent sheaf have been thoroughly investigated in \cite{FIM}; the small generalisation to the  groupoid-valued setting necessary for this paper is carried out in Section~\ref{sec.defo-sheaf}.
On the other hand, to understand the deformations of the graph of the morphism inside the direct sum $\FF \oplus \GG$, we  need  to consider the deformation problem of a pair (sheaf, subsheaf), where both the subsheaf and the sheaf deform simultaneously,  such that   the deformation of the subsheaf remains a subsheaf of the deformation of the sheaf. To find a controlling dgLa we generalised the result of \cite{DE maps}, valid for locally free sheaves, to coherent sheaves.  This is the content of Section~\ref{sec.subsheaf}.

At the level of semicosimplicial dgLas, the key ingredient is the construction of simultaneous locally free resolutions of the sheaf and of the subsheaf.   Then, we can construct  the sheaf of dgLas given by the endomorphisms of the sheaf which restrict to  endomorphisms of  the subsheaf, and the Thom--Whitney totalisation of the associated \v{C}ech complex  controls this deformation problem (see Theorem~\ref{thm.dgLa def sottofasci}). 
 Then, we apply this explicit description of the deformations of a pair (sheaf, subsheaf)  to the case  (graph of $\alpha$, $\FF \oplus \GG$).

Finally, a deformation of the morphism $\alpha: \FF \to \GG$ is obtained from the data of the three involved deformations (of $\FF$, of $\GG$ and of the pair (graph of $\alpha$, $\FF \oplus \GG$)) via a totalisation process, which corresponds to taking the homotopy limit. 
The dgLas involved are constructed from endomorphisms of bounded locally free resolutions and hence they do not vanish in  negative degree but their negative cohomology is zero.  
Therefore, we cannot directly apply Hinich's Theorem of descent of Deligne groupoids  \cite{hinich}, instead we apply our generalisation of Hinich's Theorem (Theorem~\ref{generalized hinich theorem}) to   obtain a differential graded Lie algebra that controls the infinitesimal deformations of the morphism   $\alpha: \FF \to \GG$ via the Deligne functor in groupoids.

More precisely, we fix an open affine cover ${\mathcal V} $ of $X$,  locally free resolutions $\mathcal{E}_{\mathcal{F}}^{\cdot} \to \mathcal{F} $ and $  \mathcal{E}_{\mathcal{G}}^{\cdot} \to \mathcal{G}$ of $\mathcal{F}$ and of $\mathcal{G}$, respectively, such that there exists a lifting of $\alpha\colon \FF \to \GG$ to a morphism $\mathcal{E}_{\mathcal{F}}^{\cdot} \to  \mathcal{E}_{\mathcal{G}}^{\cdot}$.
 Let $\LL^\cdot$ be the sheaf of dgLas controlling the deformation of the pair (graph of $\alpha$, $\FF \oplus \GG$).  \begin{theorems}[Theorem~\ref{thm dgla of alpha}]
	The pseudo-functor $\Def_{(\FF, \alpha, \GG)}$ of infinitesimal deformations of the morphism $\alpha:\FF \to \GG$ of coherent sheaves is equivalent to the Deligne functor associated to the Thom--Whitney dgLa $\Tot(H(\VV))$ of the following semicosimplicial dgLa $H(\VV)$:
	\[  \xymatrix{   {\Tot(\EE nd^\cdot(\EE_\FF^\cdot)({\mathcal V} ))}  \oplus {\Tot(\EE nd^\cdot(\EE_\GG^\cdot)({\mathcal V} ))}   \oplus   {\Tot(\LL^\cdot({\mathcal V} ))}\ar@<2pt>[r]\ar@<-2pt>[r] &  {\Tot(\EE nd^\cdot(\EE_\FF^\cdot\oplus \EE_\GG^\cdot)({\mathcal V} ))} }.\]
	In particular, $H^1(\Tot(H(\VV)))$ is the tangent space of $\Def_{(\FF, \alpha, \GG)}$ and  $H^2(\Tot (H(\VV)))$ is an obstruction space.
\end{theorems}

 The Thom--Whitney dgLa constructed in this way, whose quasi-isomorphism class does not depend on the choice of the cover,  is unfortunately quite complicated
to handle, but its  cohomology groups fit into a long exact sequence:
\[
\cdots \to   \! \!   	{H^i(\Tot(H(\VV)))} \! \!  \to  \! \! {\Ext_X^i(\FF, \FF)\oplus\Ext_X^i(\GG, \GG)} \! \! \to  \! \! {\Ext_X^i(\FF, \GG)} \! \! \to \! \! 
{H^{i+1}(\Tot(H(\VV)))} \! \!  \to\cdots. \]

This allows us  to more explicitly describe the tangent space to deformations of $\alpha: \FF \to \GG$ and an
obstruction space (see Corollary~\ref{cor.succ-esatta-lunga-H}). 
In a work in progress, we are aiming to  apply the techniques developed in this paper to investigate 
smoothness conditions for  the  infinitesimal deformations of $\alpha: \FF \to \GG$ and related deformation problems.

\subsection{Notation} We denote by $\kk$ a field of characteristic zero,  by $ \Art_\kk$ the category of local Artin $\kk$-algebras with residue field $\kk$, by $\Set$ the category of sets. By variety, we mean an integral (and so irreducible and reduced), separated scheme of finite type over $\kk$. We denote the tensor product over $\kk$ just by $\otimes$ instead of $\otimes_\kk$.

\subsection{Acknowledgments}  

The authors wish to thank Ruggero Bandiera, Barbara Fantechi, Domenico Fiorenza, Marco Manetti and Michael Wemyss for useful mathematical discussions and comments. 
All authors are members of  the GNSAGA group of INdAM. E.L. was partially supported by ERC Grant 101001227 (MMiMMa).
E.M. was partially supported by the European Union - Next Generation EU, Mission 4, Component 1, CUP D53D23005860006, within the project PRIN 2022L34E7W “Moduli spaces and birational geometry”.

	\section{Equivalence between gauge action and homotopy} 

		 We begin this section by recalling some definitions and properties of set-valued deformation functors, and the notions of gauge and homotopy equivalence of Maurer--Cartan elements. Then, we pass to groupoid-valued functors. We consider the   reduced Deligne functor and the notion of $2$-homotopy equivalence, which allows to construct a Deligne homotopy functor, and we prove they are equivalent. For further definitions and properties of differential graded Lie algebras, we refer the reader to \cite{Manetti.LMDT, Man Pisa}.

First we introduce the following definition. 
\begin{definition}
A \emph{differential graded Lie algebra}, briefly a \emph{dgLa}, is the data $(L,d,[\ ,\ ])$, where $L=\bigoplus_{i\in \mathbb {Z}} L^i$ is a $\mathbb  Z$-graded vector space over $\kk$, $d:L^i \rightarrow L^{i+1}$ is a linear map such that $d \circ d=0$, and $[\ ,\ ]:L^i \times L^j \rightarrow L^{i+j}$ is a bilinear map, such that:
\begin{itemize}
\item $[\ ,\ ]$ is graded skewsymmetric:  $[a,b]=-(-1)^{\deg a\deg b}[b,a]$,  
\item $[\  \,,\ ]$ verifies the graded Jacobi identity: $[a,[b,c]]=[[a,b],c]+(-1)^{\deg a\deg b}[b,[a,c]]$,
\item $[\ ,\ ]$ and $d$ verify the graded Leibniz rule: $d[a,b]=[da,b]+(-1)^{\deg a}[a,db]$, 
\end{itemize}
for every $a, b$ and $c$ homogeneous.
\end{definition}

\begin{definition}
Let $(L,d_L, [\ ,\ ]_L)$ and $(M, d_M, [\ ,\ ]_M)$ be two dgLas, a \emph{morphism of dgLas}  $\varphi:L\to M$ is a degree zero linear morphism that commutes with the brackets and the differentials.  

A \emph{quasi-isomorphism} of dgLas is a morphism of dgLas that induces an isomorphism in cohomology. 
\end{definition}

\subsection{Functors in Sets}

	The deformation functor associated to a dgLa $L$ is the functor \[\Def_L: \Art_{\kk}\to \Set,\] 
	that associates to every $A \in \Art_\kk$ the set \[\Def_L(A) := \Def(L\otimes \m_A)=\frac{\MC_L(A)}{\sim_{gauge}},\]
	where
	$$\MC_L(A):=\MC(L\otimes \m_A)=\left\{x\ \in L^1 \otimes\m_A\mid dx+\frac{1}{2}[x,x]=0\right\}$$
	is the set of the Maurer--Cartan elements, 
	and two elements $x, y \in \MC_L(A)$ are gauge equivalent if and only if there exists $a \in L^0 \otimes \m_A$ such that 
	$e^a *x=y$.   More explicitly,
	the gauge action $*$ is the action of $\exp (L^0 \otimes \m_A)$ on $\MC(L\otimes \m_A)$, given by:
	$$e^a * x= x+\sum_{n=0}^{+\infty} \frac{([a,-])^n}{(n+1)!}([a,x]-da).$$

	\begin{remark}
		Note that
		\[
		e^a*x=x \qquad \mbox{ if and only if} \qquad da+[x,a]=0.
		\]
		In particular, for any $u \in  L^{-1}\otimes \m_A$, the element $du +[x,u]$ satisfies this condition, i.e.,  $e^{du +[x,u]}*x=x$ \cite[Lemma 6.3.4]{Manetti.LMDT}. 
	\end{remark}

 Consider now the polynomial algebra $\kk[t_1, \ldots, t_n, dt_1, \ldots, dt_n]$, where the indeterminates $t_i$ have degree $0$ and the indeterminates $dt_i$ have degree $1$. The product is the obvious one, while the differential is  given by $d(t_i)=dt_i$ and $d(dt_i)=0$, for all $i$. Let $L$ be a dgLa,  then  $L[t_1, \ldots, t_n, dt_1, \ldots, dt_n]=L \otimes \kk[t_1, \ldots, t_n, dt_1, \ldots, dt_n]$  inherits the obvious structure of dgLa.

	The following results explicitly describe  the Maurer--Cartan elements in the dgLas $L[t,dt]$ and $L[t,s,dt,ds]$. 
	
	 	\medskip

	\begin{proposition} \label{pro MC_L[t,dt]}
		 Let $L$ be a dgLa and $A \in \Art_\kk$.
		Then,  every element  of
		$\MC_{L[t,dt]}(A)$ is of the form  $e^{p(t)}  * x$, where $x \in \MC_L(A)$, $p(t)\in L^0[t] \otimes
		\m_A$ and $p(0)=0$.
		Moreover,  $p(t)$ and $x$ are uniquely determined.
	\end{proposition}
	\begin{proof}
		See \cite[Corollary~7.2]{FMcone} or \cite[Lemma 6.4.8]{Manetti.LMDT}. 
	\end{proof}
	
	\begin{remark}\label{se fissa elemento allora polinomio zero}
		Let $x \in  \MC_L(A)$ and  $p(t)\in L^0[t] \otimes
		\m_A$, with $p(0)=0$ and $e^{p(t)}  * x=x$, for any $t$. Then, applying the uniqueness of the previous corollary, the polynomial $p(t) $ is the zero polynomial.
		
	\end{remark}

	\begin{proposition} \label{pro MC_L[t,s,dt,ds]}
		Let $L$ be a dgLa and $A \in \Art_\kk$.
		Then,  every element  of 
		$\MC_{L[t,s,dt,ds]}(A)$ is of the form $e^{r(t,s,ds)}  * x$, where $x \in \MC_L(A)$, and 
		\[
		r(t,s,ds)\in (L^0[t,s]\cdot t+ L^0[t,s]\cdot s+  L^{-1}[t,s]\cdot  tds) \otimes \m_A.
		\]
		Moreover,  $r(t,s,ds)$ and $x$ are uniquely determined.
	\end{proposition}

	\begin{proof}
		See \cite[Proposition 4.1]{FIM}.
	 	\end{proof}

	\smallskip
	
Next, we introduce the homotopy equivalence relation.
	
	\begin{definition}
		Let $L$ be a dgLa, $A \in \Art_\kk$ and $x,y \in \MC_L(A)$. The elements $x$ and $y$ are homotopy equivalent if and only if there exists 
		$P(t,dt) \in  \MC_{L[t,dt]}(A)$, such that $P(0)=x$ and $P(1)=y$.
	\end{definition}
	\begin{remark}\label{remark 1-omotopia esplicita}
		According to Proposition~\ref{pro MC_L[t,dt]}, $x$ and $y$ are homotopy equivalent if  and only if  there exists $p(t) \in L^0[t] \otimes
		\m_A$ with $p(0)=0$, so that $P(t,dt)= e^{p(t)} *x \in  \MC_{L[t,dt]}(A)$,  $P(0)= e^{p(0)} *x =e^0*x=x$ and  $P(1)= e^{p(1)} *x =y$.
	\end{remark}
	
	The homotopy equivalence is an equivalence relation $\sim _h$ on the set of Maurer--Cartan elements   (see  \cite[Lemma 13.1.3]{Manetti.LMDT}) and therefore we can define the following functor
	\[\Def'_L: \Art_{\kk}\to \Set,\] 
	that associates to every $A \in \Art_\kk$ the set \[\Def'_L(A) :=  \frac{\MC_L(A)}{\sim_h}.\]

	The following result establishes the well-known equivalence between  the gauge and  the homotopy relation.

\begin{proposition}\label{prop gauge=homotopy}
Let $L$ be a dgLa.
 Two Maurer--Cartan elements  are gauge equivalent if and only if they are homotopy equivalent. Then, there is an isomorphism of functors $\Def_L \cong \Def'_L $. 
	\end{proposition}
 
	\begin{proof}
	See for example \cite{SchlSt}  or \cite[Lemma 10.5.5]{Manetti.LMDT}. Here, we recall the details for further use. If $x$ and $y$ are gauge equivalent via  $a\in L^0 \otimes \m_A$, i.e., $ e^a * x = y$, then they are homotopy equivalent via $e^{at}*x \in \MC_{L[t,dt]}(A)$; indeed $e^{a\cdot0}*x=x$ and $e^{a\cdot1}*x=y$. Conversely, let $x$ and $y$ be homotopy equivalent via $P(t,dt) \in  \MC_{L[t,dt]}(A)$. Then, by Proposition~\ref{pro MC_L[t,dt]}, we have that $P(t,dt)= e^{p(t)}  * x$,  for some $p(t)\in L^0[t]\cdot t \otimes
		\m_A$. In particular, $y = P(1)=e^{p(1)}*x$ and so $x$ and $y$ are gauge equivalent via $p(1)$.
 This isomorphism is easily proved to be functorial. 
	\end{proof}
	\bigskip
	
	The aim of the following section is to extend the previous proposition to the level of groupoids (see Theorem~\ref{teo gauge= homot in gruppoidi}).
	 
	\subsection{Functors in groupoids}

	A \emph{small category} is a category whose morphisms form a set. 
	A \emph{groupoid} is a small category such that every morphism is an isomorphism. We denote the category of groupoids  by $\Grpds$.

	For every groupoid $G$, 
	the set of isomorphism classes of objects is denoted by  $\pi_0(G)$.  
	Moreover, a functor of groupoids $f:G \to G'$ is an equivalence if and only if  $f: \pi_0(G) \to    \pi_0(G')$ is a bijection and  $f: \Mor_G(g,g)\to   \Mor_{G'}(f(g),f(g))$ is an isomorphism of groups, for any $g \in G$.  
	
The category $\Grpds$ is actually a $(2,1)$-category, whose objects are groupoids, morphisms are functors of groupoids and 2-morphisms are natural transformations, which are all invertible. Pseudo-functors from $\Art_{\kk}$ to $\Grpds$, with pseudo-natural transformations as morphisms and modifications as 2-morphisms likewise form a $(2,1)$-category. 
We are often going to think of the category of groupoids inside the $(2,1)$-category of groupoids, and the category of functors $\Art_\kk \to \Grpds$ inside the $(2,1)$-category of pseudo-functors. For example, by an \emph{equivalence} of (pseudo-)functors $F, G: \Art_{\kk} \to \Grpds$ we mean a pseudo-natural transformation such that $F(A)\to G(A)$ is an equivalence of groupoids for every $A \in \Art_{\kk}$.

\medskip

	Let $L$ be a   dgLa and $A \in \Art_\kk$. We define $C_L(A)$ as the groupoid whose set of objects is $\MC_L(A)$ and whose morphisms between two objects $x$ and $y$ are defined as the set
	\[ \Mor_{C_L(A)}(x,y) = \{ e^a \in \exp(L^0 \otimes \m_A) \mid e^a * x = y\}. \]
	Note that we have a multiplication map   given by
	\[
	\Mor_{C_L(A)}(y,z) \times  \Mor_{C_L(A)}(x,y) \to  \Mor_{C_L(A)}(x,z) \qquad (e^a,e^b) \mapsto e^ae^b=e^{a\bullet b},
	\]
where  	$\bullet$ denotes the Baker--Campbell--Hausdorff product in a Lie algebra.
	The \emph{irrelevant stabiliser} of a Maurer--Cartan element $x\in\MC_L(A)$ is the normal subgroup 
	\[ I_A(x) = \{ e^{du + [x,u]} \mid u \in L^{-1} \otimes \m_A\} \subseteq\Mor_{C_L(A)}(x,x). \]
	If $a \in \exp (L^0 \otimes \m_A)$ and $e^a *x=y$, then $e^a I_A(x) e^{-a}=I_A(y)$ \cite[Lemma 6.5.5]{Manetti.LMDT}.
	This allows us to define an equivalence relation on $\Mor_{C_L(A)}(x,y) $:  two morphisms $e^a, e^b \in \Mor_{C_L(A)}(x,y)$ are equivalent if there exists an element $u \in L^{-1} \otimes \m_A$ such that $e^{a \bullet (du + [x,u])} =e^b$, or equivalently if there exists an element $v \in L^{-1} \otimes \m_A$ such that $e^{ (dv + [y,v])\bullet a} =e^b$.
	Moreover, the above multiplication is well defined on the equivalence classes and so we can define the following functor in groupoids.

	\begin{definition}\label{def Deligne groupoid}
		If $L$ is any dgLa, we define the \emph{Deligne functor} 
		\[\Del_L:\Art_\kk \to \Grpds,\]
		as the functor that associates to every $A \in \Art_\kk$ the (reduced) Deligne groupoid 
		\[\Del_L(A):= \Del(L\otimes \m_A),\]
		having as objects the Maurer--Cartan elements of $L\otimes \m_A$ and as morphisms 
		\[ \Mor_{\Del_L(A)}(x,y) = \frac{\Mor_{C_L(A)}(x,y)}{I_A(x)} = \frac{\Mor_{C_L(A)}(x,y)}{I_A(y)}.   \] 
	\end{definition}

Any morphism of dgLas $L\to M$ induces a natural transformation of functors $\Del_L\to \Del_M$.

We say that a dgLa $L$ controls a deformation problem if there is an equivalence between the functor $\Del_L$ and the deformation pseudo-functor in question.
	
\begin{remark}
We remark that the Deligne groupoid considered in \cite{hinich, bandiera-tesi, getzler} is actually the action groupoid of the gauge action on the Maurer--Cartan set, which we denote here by $C_L(A)$. What we denote with $\Del_L(A)$ is sometimes called the reduced Deligne groupoid \cite{yek}. This distinction is not needed when the dgLas considered are concentrated in non-negative degrees, as in \cite{hinich}.  
\end{remark}

\bigskip
  Here, we recall two results that will be useful later.

\begin{lemma}\cite[Lemma 2.1]{joint-defos}\label{lemma equivalenza funtori con diagramma}
Let $f\colon L\to M$ be a morphism of dgLas, $A\in \Art_{\kk}$ and let 
\[\xymatrix{\Del_L(A)\ar[r]^-{g}\ar[d]_f&G\ar[d]^{k}\\
\Del_M(A)\ar[r]^-{h}&H}\]
be a diagram in $\Grpds$.
Assume that:
\begin{enumerate}
\item $h\circ f=k\circ g$;
\item $f\colon H^i(L)\to H^i(M)$ is surjective for $i=-1$ and injective for 
$i=0$;
\item $g$ is full and essentially surjective;
\item $h$ is faithful.
\end{enumerate}
Then $g$ is an equivalence of groupoids.
\end{lemma}

\begin{remark} 
The above lemma also holds in  the (2,1)-category of groupoids,   weakening the first item  (1) and requiring that  the diagram is only $2$-commutative \cite[Remark 2.2]{joint-defos}.
\end{remark}

\begin{lemma} \label{lemma LMDT H-1 H0 H1 equivalenza gpd} \cite[Theorem 6.6.9]{Manetti.LMDT}
	Let $f: L\to M $ be a morphism of dgLas
	such that the induced map $f: H^i (L)\to H^i (M)$ is:
	\begin{enumerate}
		\item surjective for $i=-1$,
		\item bijective for $i = 0, 1$,
		\item injective for $i = 2$.
	\end{enumerate}
	Then, for every $A\in \Art_\kk$ the morphism
	\[ f: \Del_L (A)\to \Del_M(A)\]
	is an equivalence of groupoids.
\end{lemma}

\begin{remark}
Note that the equivalences of groupoids $f: \Del_L (A)\to \Del_M(A)$ are natural with respect to morphisms $A \to B$ in $\Art_\kk$ and then they induce an equivalence of functors
	\[ f: \Del_L \to \Del_M. \]

\end{remark}

	Next, we would like to define the extension of the functor $\Def'_L$ to groupoids.

	Let $L$ be a dgLa, $A \in \Art_\kk$ and $x,y \in \MC_L(A)$. Suppose that $P(t,dt)$ and $Q(t,dt)$ are two homotopies between $x$ and $y$.  According to Remark~\ref{remark 1-omotopia esplicita},   there exist unique polynomials $p(t), q(t) \in L^0[t] \cdot t\otimes
	\m_A$ such that 
	\begin{equation} 
		\begin{cases}
			P(t,dt)= e^{p(t)}  * x,   \     \ e^{p(1)}*x=y, \\
			Q(t,dt)= e^{q(t)}  * x,    \ \  e^{q(1)}*x=y. \\
		\end{cases}
	\end{equation}

	\begin{definition}\label{def.2-hom}
		In the above setup, the homotopies $P(t,dt)$ and $Q(t,dt)$ between $x$ and $y$ are $2$-homotopy equivalent if there exists  $R(t,s,dt,ds) \in \MC_{L[t,s,dt,ds]}(A) $ such that 
		\begin{equation} 
			\begin{cases}
				R(t,0,dt,0)= P(t,dt), \\
				R(t,1,dt,0)= Q(t,dt),  \\
				R(0,s,0,ds)= x, \\
				R(1,s,0,ds)= y.\\
			\end{cases}
		\end{equation}
In this case we write $P(t,dt) \sim_h Q(t,dt)$.
	\end{definition}
	\begin{remark}
		According to Propositions~\ref{pro MC_L[t,dt]} and \ref{pro MC_L[t,s,dt,ds]}, if $P(t,dt)$ and $Q(t,dt)$   are $2$-homotopy equivalent then there exists $r(t,s,ds)\in (L^0[t,s]\cdot t+ L^0[t,s]\cdot s+  L^{-1}[t,s]\cdot  tds) \otimes \m_A$, such that 
		\begin{equation} 
			\begin{cases}
				R(t,s,dt,ds)= e^{r(t,s,ds)} *x, \\
				r(t,0,0)=p(t),  \\
				r(t,1,0)= q(t), \\
				e^{r(0,s, ds) } *x  = x,\\
				e^{r(1,s, ds) } *x  = y.\\
			\end{cases}
		\end{equation}
	  Note that, since $e^{r(0,s,ds) } *x  = x$ and $r(0,0,0)=0$,  Remark~\ref{se fissa elemento allora polinomio zero} implies that $r(0,s,ds)$ is the zero polynomial. 
		
	\end{remark}
	
	\begin{example}\label{esempio omotoy p(t) con p(1) t}
		Let $x \in  \MC_L(A)$ and $p(t)  \in L^0[t]\cdot t \otimes
		\m_A$. Then, according to Proposition~\ref{pro MC_L[t,dt]},  $e^{p(t)}*x$ and $e^{p(1)t}*x$ are elements in $\MC_{L[t,dt]}(A)$. In particular, they are $2$-homotopy equivalent via 
		$e^{ p(1)ts+p(t)(1-s)}*x \in \MC_{L[t,s,dt,ds]}(A)$.
	\end{example}

	Finally, we recall this useful result.

	\begin{lemma}\label{lemma stab is irrelevant}
		Let $x(t,dt) \in \MC_{L[t,dt]}(A)$, $p(t,dt)  \in
		L[t,dt]^0\otimes \m_A$, such that $p(0)=0$, and
		$$
		e^{p(t,dt) } *x(t,dt)=x(t,dt).
		$$
		Then, $e^{p(1)} \in I_A(x(1))$, i.e., there  exists $u\in
		L^{-1}\otimes\m_A$ such that $p(1)=du+ [x(1),u]$. 
	\end{lemma}
	\begin{proof}
		See \cite[Lemma 6.15]{Dona.def maps}, \cite[Lemma 6.5.7]{Manetti.LMDT}.
	\end{proof}

	Let $L$ be a   dgLa and $A \in \Art_\kk$. We define $C^h_L(A)$ as the groupoid whose set of objects is $\MC_L(A)$ and whose morphisms between two objects $x$ and $y$ are defined as the set
	\[ \Mor_{C^h_L(A)}(x,y) = \{   P(t,dt)= e^{p(t)}  * x  \in  \MC_{L[t,dt]}(A) \mid p(t)\in L^0[t]\cdot t  \otimes \m_A \mbox{ and } e^{p(1)}*x=y\}. \]

Note that  the $2$-homotopy  defines an equivalence relation on $\Mor_{C^h_L(A)}(x,y) $  (see Lemma~\ref{lemma 2 homotopy equivalence relation}).   
Moreover, the equivalence relation is compatible with the multiplication map   given by	 
		\[
		\Mor_{C^h_L(A)}(y,z) \times  \Mor_{C^h_L(A)}(x,y) \to  \Mor_{C^h_L(A)}(x,z)
		\]
		\[(e^{q(t)}*y,e^{p(t)}*x) \mapsto  e^{q(t)\bullet p(t)}*x,
		\]
(see Lemma~\ref{lemma 2 homotopy compatibile prodotto}).  
Then, we can define the following functor in groupoids.

	\begin{definition}
		Let $L$ be any dgLa, we define the Deligne homotopy functor 
				\[\Del'_L:\Art_\kk \to \Grpds,\]
		as the functor that associates to every $A \in \Art_\kk$ the   groupoid  $\Del'_L(A)$
		having as objects the Maurer--Cartan elements of $L\otimes \m_A$ and as morphisms 
		\[ \Mor_{\Del'_L(A)}(x,y) = \frac{\Mor_{C^h_L(A)}(x,y)}{\sim_h}.   \] 
	\end{definition}
 
Any morphism of dgLas $L\to M$ induces a pseudo-natural transformation of functors $\Del'_L\to \Del'_M$.
	
Finally, we are ready to prove the extension to groupoids of Proposition~\ref{prop gauge=homotopy}.

\begin{theorem}\label{teo gauge= homot in gruppoidi}
	Let $L$ be a dgLa. Two homotopies between two Maurer--Cartan elements are $2$-homotopic if and only if the corresponding gauge elements are equivalent via the  irrelevant stabiliser. 
Then, there is an equivalence of functors 
\[\Del_L \cong \Del'_L .\]
\end{theorem}

	\begin{proof}
		
		For any $A \in \Art_\kk$, let us define  $f: \Del_L(A)\to \Del'_L(A) $ as the identity on the set of objects and as the map $e^a\mapsto e^{at}$ on morphisms. Then,  Proposition~\ref{prop gauge=homotopy} implies that  $f: \pi_0(\Del_L(A)) \to    \pi_0(\Del'_L(A))$ is a bijection.
		
		Regarding morphisms, first of all we have to verify that $f$ is well defined on the set of morphisms. Let $e^a   =  e^{a \bullet (du + [x,u])}  \in  \Mor_{\Del_L(A)}(x,y) $. Then, the equality $e^{at} * x =  e^{(a \bullet (du + [x,u]))t} *x \in 
		\Mor_{\Del'_L(A)}(x,y) $ between the homotopies holds via the $2$-homotopy equivalence  $R(t,s,dt,ds)=e^{(a \bullet (d(su) + [x,su]))t}* x \in \MC_{L[t,s,dt,ds]}(A)$.   Indeed, we have $(a \bullet (d(su) + [x,su]))t\in (L^0[t,s]\cdot t+ L^0[t,s]\cdot s+  L^{-1}[t,s]\cdot  tds) \otimes \m_A$ and 
		\begin{equation} 
			\begin{cases}
				R(t,0,dt,0)= e^{(a \bullet 0) t } *x= e^{at } *x \\
				R(t,1,dt,1)=  e^{(a \bullet (du + [x,u]))t} *x,  \\
				R(0,s,0,ds)=   e^0*x=x\\
				R(1,s,0,ds)=  e^{a \bullet (d(su) + [x,su])} *x=y.\\
			\end{cases}
		\end{equation}

		Moreover, if $e^a \in  \Mor_{\Del_L(A)}(y,z) $ and  $e^b \in  \Mor_{\Del_L(A)}(x,y) $, then $e^{(a\bullet b)t} =e^{a t\bullet bt}   \in  \Mor_{\Del'_L(A)}(x,z) $, where the equality is via a  $2$-homotopy as in Example~\ref{esempio omotoy p(t) con p(1) t}. 
 
		 Following the same  Example~\ref{esempio omotoy p(t) con p(1) t}, $f: \Mor_{\Del_L(A)}(x,x)   \to   \Mor_{\Del'_L(A)}(x,x) $ is surjective, indeed any homotopy $e^{p(t)}*x \in \MC_{L[t,dt]}(A)$ is $2$-homotopy equivalent to $e^{p(1)t}*x$, which is in the image of $f$.

		Finally, the injectivity. Suppose that  $e^{at}*x=e^{ bt} *x  \in  \Mor_{\Del'_L(A)}(x,x) $  via  $R(t,s,dt,ds) \newline =e^{r(t,s,ds)} *x \in \MC_{L[t,s,dt,ds]}(A)$. In particular, $r(1,0,0)=a$ and $r(1,1,0)=b$.
		Consider the polynomial $\mu(s,ds) =r(1,s,ds)\bullet -a \in L[s,ds]^0 \otimes\m_A$. 
		Then,
		$\mu(0,0)=r(1,0,0)\bullet -a=a\bullet-a=0$ and 
		$  e^{\mu(s,ds)} *x = e^{r(1,s,ds)\bullet -a}*x=e^{r(1,s,ds)}*x =x$, for any $s$. Therefore, applying Lemma~\ref{lemma stab is irrelevant}, $e^{\mu(1,0)} \in I_A(x)$, i.e., there  exists $u\in
		L^{-1}\otimes\m_A$ such that $\mu(1,0)=du+ [x,u]$, which is equivalent to $ b \bullet -a = du+ [x,u]$.
		
 Finally, it is obvious that $f$ is natural with respect to morphisms in $\Art_\kk$ and it defines an equivalence of functors. 
	\end{proof}

	\section{Descent of Deligne groupoids}\label{sezione Descent of Deligne groupoids}

	 In the first part of this section, we recall the definitions of a semicosimplicial object, the total groupoid of a semicosimplicial groupoid and the Thom--Whitney totalisation of a semicosimplicial dgLa. In the second part, we recall some useful results about Deligne groupoids and truncations of semicosimplicial dgLas. Finally, we prove in Theorem~\ref{generalized hinich theorem} a generalisation of  Hinich's Theorem on descent of Deligne groupoids \cite{hinich} (see Theorem~\ref{teorema di hinich}) and of \cite[Theorem 4.10]{FIM} (see Theorem~\ref{teorema FIM}), valid for any semicosimplicial dgLa $\g^{\Delta}$ with non-negative cohomology.

	\subsection{Semicosimplicial objects and total constructions}
	
	A semicosimplicial object  $A^\Delta$ in a category $C$ is a diagram
	
	\[
	\xymatrix{ A_0
		\ar@<2pt>[r]\ar@<-2pt>[r] & A_1
		\ar@<4pt>[r] \ar[r] \ar@<-4pt>[r] &  A_2
		\ar@<6pt>[r] \ar@<2pt>[r] \ar@<-2pt>[r] \ar@<-6pt>[r]&
		\cdots},
	\]
	where each $A_i$ is an object in $C$, and for each
	$i>0$, there are $i+1$ morphisms  in $C$
	\[
	\partial_{k,i}\colon A_{i-1}\to A_{i},
	\qquad k=0,\dots,i,
	\]
	such that
	$\partial_{k+1,i+1}\partial_{j,i}=\partial_{j,i+1}\partial_{k,i}$,
	for any
	$k\geq j$.

	\begin{definition} \label{def. total grpds}
		Let $G^\Delta$ be a semicosimplicial groupoid. The \emph{total groupoid} $\Tot(G^\Delta)$ of  $G^\Delta$ is the groupoid defined as follows:
		\begin{itemize}
			\item the objects of $\Tot(G^\Delta)$ are the elements of the form $(l,m)$, where $l$ is an object in $G_0$ and $m: \partial_{0,1} l \to \partial_{1,1} l$ is a morphism in $G_1$ such that the following diagram commutes  in $G_2$

			\[ \xymatrix{ &     \partial_{0,2}\partial_{0,1} l  \ar@{=}[dl]   \ar[dr]^{\partial_{0,2} m}     &  \\
				\partial_{1,2}\partial_{0,1} l  \ar[d]^{\partial_{1,2} m}  &         &  \partial_{0,2}\partial_{1,1} l \ar@{=}[d]    \\
				\partial_{1,2}\partial_{1,1} l\ar@{=}[dr]  &  &\partial_{2,2}\partial_{0,1} l \ar[dl]^{\partial_{2,2} m}  \\
				&  \partial_{2,2}\partial_{1,1} l,  &   }
			\]

			\item the morphisms between two objects $(l_0,m_0)$ and $(l_1,m_1)$ are the morphisms $a$ in $G_0$ such that the following diagram commutes
			\[ \xymatrix{ \partial_{0,1} l_0    \ar[r]^{m_0}\ar[d]_{\partial_{0,1}a} & \partial_{1,1}  l_0 \ar[d]^{\partial_{1,1}a} \\  
				\partial_{0,1} l_1   \ar[r]_{m_1} &\partial_{1,1}  l_1.  }   \] 
		\end{itemize}
	\end{definition}

	\begin{example}\label{esempio Deligne simpliciale} {\bf{Totalisation of a semicosimplicial Deligne groupoid}. } 
		Let  ${\mathfrak g}^\Delta$ be a semicosimplicial dgLa. Then, for any $A \in \Art_\kk$, we can consider  
		the semicosimplicial groupoid
		\[
		\Del_{\g^{\Delta}}(A): \xymatrix{  \Del_{\g_0}(A)
			\ar@<2pt>[r]\ar@<-2pt>[r] &  \Del_{\g_1}(A)
			\ar@<4pt>[r] \ar[r] \ar@<-4pt>[r] &  \Del_{\g_2}(A)
			\ar@<6pt>[r] \ar@<2pt>[r] \ar@<-2pt>[r] \ar@<-6pt>[r]&
			\cdots}.
		\]
		In this case, an object of  $\Tot(\Del_{\g^{\Delta}}(A))$ is given by a pair $(l,e^m)$, where $l\in \MC_{\g_0}(A)$ and $m \in {\g_1}^0\otimes \m_A$, such that 
		$e^m* \partial_{0,1} l = \partial_{1,1} l$ (note that a morphism $e^m$ is actually a class up to irrelevant stabilisers). The commutativity in  $\Del_{\g_2}(A)$ is equivalent to the existence  
		of an element $u \in  {\g_2}^{-1}\otimes \m_A$ such that  ${\partial_{0,2}m} \bullet {-\partial_{1,2}m} \bullet
		{\partial_{2,2}m} =du+[\partial_{2,2}\partial_{0,1}l,u]$.
		
		A morphism  between $(l_0,e^{m_0})$ and $(l_1,e^{m_1})$ is given by an element $a \in  {\g_0^0}\otimes \m_A$ such that $e^a*l_0=l_1$ (note that also here $e^a$ is a class up to irrelevant stabiliser). The commutativity   
		is now equivalent to the existence of an element $b \in  {\g_1}^{-1}\otimes \m_A$ such that 
		$- m_0\bullet -\partial_{1,1}a \bullet m_1
		\bullet \partial_{0,1}a=db+[\partial_{0,1}l_0,b].$  	

This defines the following functor in groupoids: 
\[  \Tot(\Del_{\g^{\Delta}}): \Art_\kk \to \Grpds \]
that associates, to every $A \in \Art_\kk$, the groupoid $\Tot(\Del_{\g^{\Delta}}(A))$ defined as above. 
\end{example}

	\bigskip
	
	For every $n\ge 0$, we denote by $\Omega_n$ the differential graded
	commutative algebra of polynomial differential forms on the
	standard $n$-simplex $\Delta^n$:
	\[ \Omega_n=\frac{\kk[t_0,\ldots,t_n,dt_0,\ldots,dt_n]}{(\displaystyle \sum_{i=0}^n t_i-1,\displaystyle \sum_{i=0}^n dt_i)}.\]

	\begin{definition} \label{def.TW dgLA}
		Let $\g^{\Delta}$ be a semicosimplicial dgLa.
		The \emph{Thom--Whitney dgLa} associated to $\g^{\Delta}$ is the dgLa 
		\[\operatorname{Tot}(\mathfrak{g}^\Delta) =\{ (x_n)_{n}\in \prod_n \Omega_n\otimes {\mathfrak g}_n
		\mid \delta^{k,n}x_n= \partial_{k,n}x_{n-1}\quad \forall\; 0\le k\le n\},\]
		where, for $k=0,\ldots,n$, $\delta^{k,n}\colon \Omega_n\to \Omega_{n-1}$ are the face maps and $\partial_{k,n}\colon \g_{n-1}\to \g_n$ are the maps of the semicosimplicial dgLa $\g^{\Delta}$. 
	\end{definition}
 
\begin{remark} \label{remark quasi iso semicosimpliciale} We recall that there exists a quasi-isomorphism of dg vector spaces 
	between  $\operatorname{Tot}(\mathfrak{g}^\Delta)$ and the total complex associated to $\g^{\Delta}$. Moreover, let $f:  \g^{\Delta} \to \h^{\Delta}$ be a morphism of  semicosimplicial dgLas, such that $f_i:  \g_i \to \h_i $ is a quasi-isomorphism of dgLas for every $i$, then it induces a quasi-isomorphism $\operatorname{Tot}(\mathfrak{g}^\Delta)\to \operatorname{Tot}(\mathfrak{h}^\Delta)$  of dgLas. We refer to \cite[Chapter 7]{Manetti.LMDT}  for more details and properties.  \end{remark}

\begin{example}
Let ${\mathcal L}^\cdot$ be a sheaf of dgLas over a  variety $X$ and $\mathcal{V}=\{V_i\}_i$ an open affine cover of $X$. The \v{C}ech semicosimplicial  dgLa  is defined as:
\[
\xymatrix{ {\mathcal L^\cdot}({\mathcal V}) \ : \ \  {\prod_i\mathcal{L}^\cdot(V_i)}
\ar@<2pt>[r]\ar@<-2pt>[r] & {
\prod_{i,j}\mathcal{L}^\cdot(V_{ij})}
      \ar@<4pt>[r] \ar[r] \ar@<-4pt>[r] &
      {\prod_{i,j,k}\mathcal{L}^\cdot(V_{ijk})}
\ar@<6pt>[r] \ar@<2pt>[r] \ar@<-2pt>[r] \ar@<-6pt>[r]& \cdots},
\]
where, as usual, $V_{ij}=V_i \cap V_j$, $V_{ijk}=V_i \cap V_j \cap V_k$, $\mathcal{L}^\cdot(V_i)=\Gamma(V_i, \mathcal{L}^\cdot)$  and  the morphisms are the restriction maps. 
Here, the  total complex associated to
 $\mathcal{L}^\cdot(\mathcal{V})$ is  the \v{C}ech
complex $\check{C}(\mathcal{V},\mathcal{L}^\cdot)$ of the sheaf $\mathcal{L}^\cdot$. In particular, by the quasi-isomorphism between the total complex and the Thom--Whitney dgLa, for all $k \in \ZZ$, we have  the following isomorphisms
\[H^k(\Tot(\mathcal{L}^\cdot(\mathcal{V}))) \cong  \check{H}^k(\check{C}(\mathcal{V},\mathcal{L}^\cdot))\cong     \mathbb{H}^k(X, \mathcal{L}^\cdot).
\]
Note that the  quasi-isomorphism class of the   Thom--Whitney dgLa  does not depend on the choice of the affine cover.
\end{example}

\subsection{Truncations of semicosimplicial dgLas and useful results}
In this part, we recall some useful results which we will need in the proof of descent of Deligne groupoids.

According to \cite[Section 4]{FIM}, for any semicosimplicial dgLa $ \mathfrak{g}^{\Delta}$
\[ \mathfrak{g}^{\Delta}: \ \  \xymatrix{ \mathfrak{g}_0
		\ar@<2pt>[r]\ar@<-2pt>[r] & \mathfrak{g}_1
		\ar@<4pt>[r] \ar[r] \ar@<-4pt>[r] & \mathfrak{g}_2
		\ar@<6pt>[r] \ar@<2pt>[r] \ar@<-2pt>[r] \ar@<-6pt>[r] &
		\ldots }
	\]
	and $i \in\mathbb{N} \cup
	\{\infty\}$, we denote  by
	$\g^{\Delta_{[0,i]}}$,   the   semicosimplicial dgLa defined by
	\[
	(\g^{\Delta_{[0,i]}})_n=
	\begin{cases}
		\g_n &\text{for }  0\leq n\leq i\\
		0&\text{otherwise},
	\end{cases}
	\]
	with the obvious maps.

	For any    $i_1\leq i_2 \in\mathbb{N} \cup
	\{\infty\}$, the map
	$\Id_{[0,i_1]}\colon \g^{\Delta_{[0,i_2]}}\to
	\g^{\Delta_{[0,i_1]}}$ given by
	\[
	\Id_{[0,i_1]}\biggr\vert_{(\g^{\Delta_{[0,i_2]}})_n}=
	\begin{cases}
		\Id_{\g_n}&\text{if } 0 \leq n \leq i_1\\
		0&\text{otherwise}.
	\end{cases}
	\]
	is a morphism of semicosimplicial dgLas that  induces  the natural morphism of dgLas
	$\Id_{[0,i_1]}:\Tot (\g^{\Delta_{[0,i_2]}})\to
	\Tot (\g^{\Delta_{[0,i_1]}})$.

		\begin{proposition} \label{prop.tre suff}
		Let $\g^{\Delta}$ be a semicosimplicial dgLa, such that
		$H^j(\g_i)=0$ for all $i\geq 0$ and $j<0$. Then, the morphism
		$\Id_{[0,2]}: \Tot (\g^\Delta) \to
		\Tot (\g^{\Delta_{[0,2]}})$ induces an equivalence of functors:
		\[
		\Del_{\Tot (\g^\Delta) } \xrightarrow{\sim}
		\Del_{\Tot (\g^{\Delta_{[0,2]}})}.
		\]
 	
	\end{proposition}

	\begin{proof}
	This is a generalisation of \cite[Proposition 3.8]{FIM} for functors in groupoids. 
		By looking at the spectral sequences associated to the double complexes of $\g^\Delta$ and $\g^{\Delta[0,2]}$ and taking into account the hypothesis that $H^j(\g_i)=0$ for all $i\geq 0$ and $j<0$, it is easy to check that the morphism of dgLas
		\[\Id_{[0,2]}: \Tot (\g^\Delta) \to
		\Tot (\g^{\Delta_{[0,2]}})\]
		satisfies the hypotheses of Lemma~\ref{lemma LMDT H-1 H0 H1 equivalenza gpd} and then it induces the required equivalence of functors.
	\end{proof}
 
Next, let us describe the Maurer--Cartan elements of $\Tot (\g^{\Delta_{[0,1]}})$ and $\Tot (\g^{\Delta_{[0,2]}})$.
 
	 	\begin{remark}\label{rem.elemMC-tronc}
		According to   Propositions~\ref{pro MC_L[t,dt]} and \ref{pro MC_L[t,s,dt,ds]} and \cite[Proposition 4.1]{FIM},  for any  $A \in \Art_\kk$,
		the  elements in $\MC_{\Tot (\g^{\Delta_{[0,1]}}) }(A)$
		are of the form
		\[
		(x, e^{p(t)}* \partial_{0,1}x),
		\]
		where $x \in \MC_{\g_0}(A)$,  $p(t) \in \g_1^0[t]\cdot t\otimes \m_A$, are uniquely determined and satisfy $\partial_{1,1}x=e^{p(1)}*\partial_{0,1}x$.
		The   elements in $\MC_{\Tot (\g^{\Delta_{[0,2]}}) }(A)$
		are of the form
		\[
		(x, e^{p(t)}* \partial_{0,1}x, e^{ r(t,s,ds)}*\partial_{0,2}\partial_{0,1}x),
		\]
		  where $x \in \MC_{\g_0}(A)$,  $p(t) \in \g_1^0[t]\cdot t\otimes \m_A$,   $r(t,s,ds) \in (\g_2^0[t,s]\cdot t+ \g_2^0[t,s]\cdot s +\g_2^{-1}[t,s]\cdot tds) \otimes \m_A$  are uniquely determined and they satisfy		
		
		\begin{equation}  \label{eq.face conditions}
			\begin{cases}
				\partial_{1,1}x=e^{p(1)}*\partial_{0,1}x, \\
				\partial_{0,2}p(t)= r(0,t,dt),\\
				\partial_{1,2}p(t)= r(t,0,0),   \\
				e^{(- \partial_{2,2}p(t))\bullet(r(t,1-t,dt))\bullet(-r(0,1))}*\partial_{2,2}\partial_{0,1}x=\partial_{2,2}\partial_{0,1}x.
			\end{cases}
		\end{equation}

	\end{remark}

\subsection{Proof of descent of Deligne groupoids}

As preliminary steps, we investigate the behaviour of  the Deligne functors of the truncations $\g^{\Delta_{[0,1]}}$ and $\g^{\Delta_{[0,2]}}$.

\bigskip

Generalising \cite[Proposition 4.2]{FIM}, we introduce the pseudo-natural transformation:
\[ \Phi_1: \Del_{\Tot (\g^{\Delta_{[0,1]}}) }   \to   \Tot( \Del_{\g^{\Delta_{[0,1]}}}),\]
defining, for every $A \in \Art_\kk$, the functor of groupoids 
\[ \Phi_1: \Del_{\Tot (\g^{\Delta_{[0,1]}}) } (A)  \to   \Tot( \Del_{\g^{\Delta_{[0,1]}}}(A)) ,\]
as follows.

An object in $ \Del_{\Tot(\g^{\Delta[0,1]})}(A) $ is a Maurer--Cartan element in   $\Tot(\g^{\Delta[0,1]})(A)$ and so it is of the form $(x, e^{p(t)} * \partial_{0,1}x)$, where $x \in \MC_{\g_0}(A)$, as in Remark~\ref{rem.elemMC-tronc}. We define 
\[\Phi_1 (x, e^{p(t)} * \partial_{0,1}x) = (x,e^{p(1)}).\] 

For a morphism  $z(\xi, d\xi)= \left( e^{T(\xi)}* x_0, e^{U(t,dt; \xi)}* e^{p(t)}*\partial_{0,1} x_0   \right)$  between $\eta_0=(x_0, e^{p_0(t)} * \partial_{0,1}x_0)$ and  $\eta_1=(x_1, e^{p_1(t)} *\partial_{0,1}x_1)$, we define 
\[\Phi_1( z(\xi, d\xi)) = e^{T(1)}\] 
that is a morphism in  $\Tot( \Del_{\g^{\Delta_{[0,1]}}}(A))$ between $(x_0,e^{p_0(1)})$ and $(x_1,e^{p_1(1)})$  (see \cite[Proposition 4.2]{FIM}). 
It is easy to check that it is a pseudo-natural transformation of functors.

\begin{remark} Note that $\Phi_1$ is actually well-defined. According to Theorem~\ref{teo gauge= homot in gruppoidi}, a morphism in $ \Del_{\Tot(\g^{\Delta[0,1]})}(A) $  is a class of homotopies up to $2$-homotopies. Consider two morphisms  between $\eta_0$ and $\eta_1$: $z_0(\xi, d\xi)= \left( e^{T_0(\xi)}* x_0, e^{U_0(t,dt; \xi)}* e^{p_0(t)}*\partial_{0,1} x_0   \right)$  and $z_1(\xi, d\xi)= \left( e^{T_1(\xi)}* x_0, e^{U_1(t,dt; \xi)}* e^{p_0(t)}*\partial_{0,1} x_0   \right)$, such that there exists a $2$-homotopy
		$R(\xi, \lambda, d\xi,  d\lambda)$ between them. Applying Proposition~\ref{pro MC_L[t,s,dt,ds]} with $L=\Tot(\g^{\Delta_{[0,1]}})$,  the first component  is of the form  $ e^{r_0(\xi, \lambda,  d\lambda)}* x_0$.
		In particular, this implies that $e^{r_0(1,\lambda, d\lambda)}* x_0 =x_1 $ for all $\lambda$, so that $e^{-T_0(1)\bullet r_0(1,\lambda, d\lambda)}* x_0 = e^{-T_0(1)}* x_1 =x_0$.  Since $-T_0(1)\bullet r_0(1,0, 0) =0$,    Lemma~\ref{lemma stab is irrelevant} implies that
		$e^{-T_0(1)\bullet r_0(1,1, 0)}= e^{-T_0(1)\bullet T_1(1)}$ belongs to the irrelevant stabiliser of $x_0$, and thus
		$\Phi_1( z_0(\xi, d\xi)) = e^{T_0(1)}$ and $\Phi_1( z_1(\xi, d\xi)) = e^{T_1(1)}$ are equivalent morphisms in  $\Tot( \Del_{\g^{\Delta_{[0,1]}}}(A))$.
		
\end{remark}

	\begin{proposition} \label{propo equivalenza in[0,1]}
Let $\g^\Delta$ be a semicosimplicial dgLa such that $H^j(\g_i)=0$ for all $i\geq 0$ and $j<0$. Then, 
		\[ \Phi_1: \Del_{\Tot (\g^{\Delta_{[0,1]}}) }   \to   \Tot( \Del_{\g^{\Delta_{[0,1]}}})\]
is an equivalence of functors. 

	\end{proposition}

	\begin{proof} This is a generalisation of   \cite[Proposition 4.3]{FIM} to groupoids. 
		For any $A \in \Art_\kk$, consider the diagram of groupoids
		\[ \xymatrix{  \Del_{\Tot(\g^{\Delta_{[0,1]}})}(A) \ar[r]^{\Phi_1} \ar[d]^{\Id_{[0,0]}} &    \Tot( \Del_{\g^{\Delta_{[0,1]}}}(A))\ar[d]^{\Id_{[0,0]}}  \\
			\Del_{\g_0}(A) \ar[r]^{\Id} &  \Del_{\g_0}(A),}\] 
			that is obviously commutative on the objects and  morphisms.

The morphism $\Phi_1$ is essentially surjective: for any object  $(l, e^{ m})$ in 
		 $\Tot( \Del_{\g^{\Delta_{[0,1]}}}(A))$ there exists  the Maurer--Cartan element  $( l, e^{mt}* \partial_{0,1}l) \in\Tot(\g^{\Delta[0,1]})(A)$ such that  $\Phi_1 ( l, e^{mt}* \partial_{0,1}l)=(l, e^{ m})$.
	 
		Let $a$ be a morphism between $(l_0,e^{m_0})$ and $(l_1,e^{m_1})$ in $\Tot( \Del_{\g^{\Delta_{[0,1]}}}(A))$ and let    $ b \in \g_1^{-1} \otimes \m_A$  such that $- m_0\bullet -\partial_{1,1}a \bullet m_1
		\bullet \partial_{0,1}a=db+[\partial_{0,1}l_0,b].$  
Then, the element
\begin{equation}\label{equazione esplicita omotopia che solleva gauge}
			z(\xi, d\xi) = \left(e^{\xi a}* l_0, e^{t\left(\partial_{11}(\xi a) \bullet m_0 \bullet (d(\xi b) + [\partial_{01}l_0, \xi b]) \bullet - \partial_{01}(\xi a)\right) \bullet \partial_{01}(\xi a)} * \partial_{01}l_0 \right);  
		\end{equation}
		is a homotopy between  the liftings $(l_0,e^{tm_0}* \partial_{01}l_0)$ and $(l_1,e^{tm_1}* \partial_{01}l_1)$ in  $ \Del_{\Tot(\g^{\Delta[0,1]})}(A)$
		(see \cite[Proposition 4.3]{FIM}), such that   $\Phi_1( z(\xi, d\xi)) =a$; therefore,  $\Phi_1$ is full.
		
		Finally, consider the morphism  $\Id_{[0,0]}:  \g^{\Delta_{[0,1]}}\to
\g^{\Delta_{[0,0]}} =\g_0$, that is the identity on $\g_0$ and zero otherwise, and the induced morphism $\Id_{[0,0]}: 
\Tot (\g^{\Delta_{[0,1]}})  \to \Tot (\g^{\Delta_{[0,0]}}) =\g_0 $. Note that $ H^{-1}(\Tot (\g^{\Delta_{[0,1]}}) )=H^{-1}(  \g_0)= 0$ and  
$ \Id_{[0,0]}: H^{0}(\Tot (\g^{\Delta_{[0,1]}}) ) \to H^{0}(  \g_0)$ is injective, so applying Lemma~\ref{lemma equivalenza funtori con diagramma}, we conclude that $\Phi_1$ is an equivalence of groupoids for every $A\in \Art_\kk$. Moreover, as noticed it is a pseudo-natural transformation, then $\Phi_1$ is an equivalence of functors as required.
 \end{proof}

Next, we would like to extend the previous result to the dgLa 	$\Tot (\g^{\Delta_{[0,2]}}) $. Generalising  \cite[Proposition 4.5]{FIM}, we introduce the pseudo-natural transformation:
 \[ \Phi_2: \Del_{\Tot (\g^{\Delta_{[0,2]}}) }   \to \Tot( \Del_{\g^{\Delta}})\]
defining, for every $A \in \Art_\kk$, the functor of groupoids 
 \[ \Phi_2: \Del_{\Tot (\g^{\Delta_{[0,2]}}) } (A)  \to \Tot( \Del_{\g^{\Delta}}(A))\]
as follows.

According to Remark~\ref{rem.elemMC-tronc}, we define $\Phi_2$ on objects as
\[
\Phi_2( x, e^{p(t) }* \partial_{0,1} x, e^{r(t,s,ds) }* \partial_{0,2}\partial_{0,1}x ) =  (x, e^{p(1)}),
\]
which is well-defined by \cite[Proposition 4.5]{FIM}. 
A   morphism  $z(\xi, d\xi)$ between $\eta_0=(x_0, e^{p_0(t)} * \partial_{0,1}x_0, e^{r_0(t,s,ds) }* \partial_{0,2}\partial_{0,1}x_0)$ and  $\eta_1=(x_1, e^{p_1(t)} *\partial_{0,1}x_1, e^{r_1(t,s,ds) }* \partial_{0,2}\partial_{0,1}x_1)$, has the form
\[ z(\xi, d\xi)
= 	\left(e^{ T(\xi)}*x, e^{U(t,dt,\xi,d\xi)}e^{p(t) } * \partial_{0,1}x, e^{V(t,s,\xi,dt,ds,d\xi)}e^{r(t,s,ds) }* \partial_{0,2}\partial_{0,1}x\right);\]  
thus, we define 
\[\Phi_2( z(\xi, d\xi)) = e^{T(1)}.\]
Since the image involves only $T(\xi)$, it is possible to prove that $\Phi_2$ is well defined on morphisms using the same arguments used for $\Phi_1$.
It is easy to check that it is a pseudo-natural transformation of functors. 

	\begin{proposition} \label{propo equivalenza in[0,2]}
		Let $\g^\Delta$ be a semicosimplicial dgLa such that $H^j(\g_i)=0$ for all $i\geq 0$ and $j<0$.
		Then, 
	\[ \Phi_2: \Del_{\Tot (\g^{\Delta_{[0,2]}}) }   \to \Tot( \Del_{\g^{\Delta}})\]
		 is an equivalence of functors.
		
	\end{proposition}

	\begin{proof} 
		We apply Lemma~\ref{lemma equivalenza funtori con diagramma} to the diagram of groupoids:
 		\[ \xymatrix{  \Del_{\Tot(\g^{\Delta_{[0,2]}})}(A) \ar[r]^{\Phi_2} \ar[d]^{\Id_{[0,1]}} &   \Tot( \Del_{\g^{\Delta}}(A))   \ar[d]^{\Id_{[0,1]}} \\
			\Del_{\Tot(\g^{\Delta_{[0,1]}})}(A) \ar[r]^{\Phi_1} &  \Tot( \Del_{\g^{\Delta_{[0,1]}}}(A)) ,} \]
that clearly commutes on objects and morphisms.

Because of Proposition~\ref{propo equivalenza in[0,1]}, $\Phi_1$ is faithful.

By looking at the spectral sequences associated to the double complexes of $\g^\Delta$ and $\g^{\Delta_{[0,2]}}$ and taking into account the hypothesis that $H^j(\g_i)=0$  for all $i\geq 0$ and $j<0$, it is easy to check that $H^0(\Tot(\g^{\Delta_{[0,1]}}))=H^0(\Tot(\g^{\Delta_{[0,2]}}))$ and $H^{-1}(\Tot(\g^{\Delta_{[0,1]}}))= H^{-1}(\Tot(\g^{\Delta_{[0,2]}})) =0$. Then the morphism of dgLas 
			\[\Id_{[0,1]}:  \g^{\Delta_{[0,2]}} \to \g^{\Delta_{[0,1]}}\]
			satisfies hypothesis (2) of Lemma~\ref{lemma equivalenza funtori con diagramma}.

The functor $\Phi_2$ is essentially surjective by \cite[Proposition 4.6]{FIM}.
			
To prove it is full, let $\eta_0$ and $\eta_1$ be two objects in $\Del_{\Tot(\g^{\Delta_{[0,2]}})}(A)$,  they can be written as
			\[   \eta_i=\left(x_i, e^{p_i(t) }* \partial_{01} x_i, e^{ r_i(t,s,ds) }* \partial_{02}\partial_{01}x_i \right)  \]
			where the elements $x_i, p_i, r_i$ are uniquely determined and satisfy the relations in \eqref{eq.face conditions}.
			Let $a$ be a morphism in  $\Tot( \Del_{\g^{\Delta}}(A))$ between $\Phi_2(\eta_0)= (x_0, e^{p_0(1)})$ and $\Phi_2(\eta_1)=(x_1, e^{p_1(1)})$.  The aim is to construct a morphism $\widetilde{z}$  between $\eta_0$ and $\eta_1$ in  $\Del_{\Tot(\g^{\Delta_{[0,2]}})}(A)$, such that $\Phi_2(\widetilde{z})=a$. 
			
			Combining the homotopy in Equation \eqref{equazione esplicita omotopia che solleva gauge}  and Example~\ref{esempio omotoy p(t) con p(1) t} (see also  \cite{FIM}), it is possible to construct a homotopy $z(\xi, d\xi)$ between $(x_0, e^{p_0(t)}* \partial_{01} x_0)$ and $(x_1, e^{p_1(t)}* \partial_{01} x_1)$ in $\Del_{\Tot(\g^{\Delta[0,1]})}(A)$, such that $\Phi_2(z) =a$. 
			
 The conclusion follows by the next lemma and by noting that $\Phi_2$ is a pseudo-natural transformation of functors.
	\end{proof}

	\begin{lemma}\cite[Proposition 4.9]{FIM}
		Let $(x_0, x_1, x_2), (x'_0, x'_1, x'_2)$ be Maurer--Cartan elements of the dgLa $\Tot(\g^{\Delta_{[0,2]}})(A)$. Suppose the corresponding Maurer--Cartan element $(x_0,x_1), (x_0', x_1')$ in  $\Tot(\g^{\Delta_{[0,1]}})(A)$ are gauge  equivalent via $(a_0, a_1) \in \Tot^0(\g^{\Delta_{[0,1]}})(A)$. 
		Then, $(x_0, x_1, x_2)$ and $(x'_0, x'_1, x'_2)$ are gauge equivalent in $\Tot(\g^{\Delta_{[0,2]}})(A)$ too. 
		
	\end{lemma}

	 \bigskip
	 	
 Finally, we are able to generalise Hinich's Theorem on descent of Deligne groupoids of  \cite{hinich} and \cite{FIM}. For the reader's convenience, we restate them in our notation.  
		
		\begin{theorem}[{\cite[Section 4]{hinich}}]\label{teorema di hinich}
			Let $\g^{\Delta}$ be a semicosimplicial dgLa such that $\g^j_i=0$ for all $i \geq 0$ and all $j <0$.
Then there is an equivalence of functors  \[ \Del_{\Tot (\g^{\Delta }) }   \to    \Tot(\Del_{\g^{\Delta }}).\]
		\end{theorem}

\begin{theorem}[{\cite[Theorem 4.10]{FIM}}]\label{teorema FIM}
			Let $\g^{\Delta}$ be a semicosimplicial dgLa with  $H^j(\g_i) = 0$ for all $i \geq 0$ and $j < 0$. Then there is a natural isomorphism of functors in sets    $\pi_0(\Del_{\Tot (\g^{\Delta})}) \cong  \pi_0(\Tot(\Del_{\g^{\Delta}}))$.
		\end{theorem}

The next theorem is our generalisation. 

	\begin{theorem} \label{generalized hinich theorem}
		Let $\g^\Delta$ be a semicosimplicial dgLa such that $H^j(\g_i)=0$ for all $i\geq 0$ and $j<0$.
		Then, there is an equivalence of functors
		\[ \Phi: \Del_{\Tot (\g^{\Delta }) }   \to   \Tot(\Del_{\g^{\Delta }}).\]
 
	\end{theorem}
	\begin{proof}
		By  Proposition~\ref{prop.tre suff}, $ \Del_{\Tot (\g^{\Delta }) }$ is equivalent to  $\Del_{\Tot (\g^{\Delta_{[0,2]}})}$ and then, by Proposition~\ref{propo equivalenza in[0,2]}, it is equivalent to   $ \Tot(\Del_{\g^{\Delta }}) $.  
	\end{proof}

\begin{remark}\label{rem.hyp-ott}
	It is easy to construct examples of semicosimplicial dgLas $\g^\Delta$ with non-trivial negative cohomology such that the previous result does not hold: for instance, a semicosimplicial  dgLa $\g^\Delta$  vanishing everywhere except for $\g_2$ in degree $-1$. However, by looking at the spectral sequences associated to $\g^{\Delta_{[0,0]}}$, $\g^{\Delta_{[0,1]}}$, $\g^{\Delta_{[0,2]}}$ and $\g^\Delta$ it is easy to see that all our proofs work for semicosimplicial dgLas $\g^\Delta$ such that $H^{-n}(\g_n)=0$ for all $n \geq 1$, $H^{-n+1}(\g_n)=0$ for all $n \geq 2$ and $H^{-n+2}(\g_n)=0$ for all $n \geq 3$.
	\end{remark}

\section{Review of deformations of a coherent sheaf}\label{sec.defo-sheaf}

In this section, we review the infinitesimal deformations of a coherent sheaf, extending  the computations in \cite[Section 2]{FIM} to the level of groupoids. 

Let $X$ be a smooth  variety over $\mathbb{K}$  and  $\FF$ a coherent sheaf of $\OO_X$-modules over $X$.

\begin{definition}

An \emph{infinitesimal deformation of $\mathcal{F}$ over $A\in \Art_{\kk}$} is given by a
coherent  sheaf $\mathcal{F}_A$ of $\OO_X\otimes A$-modules on
$X\times \Spec A$, flat over $A$, with a morphism of sheaves $\pi_A:
\mathcal{F}_A \to \mathcal{F}$ inducing an isomorphism
$\mathcal{F}_A\otimes_A \kk \cong \mathcal{F}$.
Two deformations $\mathcal{F}_A$ and $ \mathcal{F}_A'$ of $\mathcal{F}$ over $A$ are isomorphic if there exists an
isomorphism of sheaves $f: \mathcal{F}_A \to \mathcal{F}_A'$ that
commutes with the morphisms to $\mathcal{F}$.

\end{definition}

\begin{definition} 
The \emph{pseudo-functor of infinitesimal deformations of the sheaf $\FF$} is defined as
 	\[ \Def_{\FF} : \Art_\kk \to \Grpds,\]
 	that associates, to any $A \in \Art_\kk$, the groupoid $\Def_{\FF}(A)$, whose objects are the infinitesimal deformations of the sheaf $\FF$ over
 	$A$ and whose morphisms are the isomorphisms of them.
  Note that  
\[ \pi_0(\Def_{\FF}(A)) = \{\mbox{isomorphism classes of infinitesimal deformations of the sheaf $\FF$ over $A$}\} \]
is the classical functor of deformations in $\Set$.

\end{definition} 

The Kodaira--Spencer approach to infinitesimal deformations of
$\mathcal F$ consists in deforming the sheaf $\mathcal F$ locally
in such a way that local deformations glue together to a global
deformation of the sheaf $\mathcal{F}$. 
Since they are flat deformations, 
we can investigate them  using a bounded locally free resolution $ (\EE_\FF^{\cdot},d)$ of $\mathcal F$ (see \cite[Section 1.3]{Artin} or \cite[Theorem A.10]{Sernesi} for details of these correspondences). Let
\[ (\EE_\FF^{\cdot},d): \qquad\qquad  0\longrightarrow  \EE_\FF^{-n}
 \stackrel{d}{\longrightarrow } \cdots\stackrel{d}{\longrightarrow }
  \EE_\FF^{-1} \stackrel{d}{\longrightarrow }\EE_\FF^0 \longrightarrow 0\]
be a bounded complex of locally free sheaves such that  the following sequence is exact: 
\[0 \longrightarrow \EE_\FF^{-n} \stackrel{d}{\longrightarrow }
\cdots \stackrel{d}{\longrightarrow }  \EE_\FF^{-1}
\stackrel{d}{\longrightarrow }  \EE_\FF^0
\stackrel{d}{\longrightarrow } \mathcal F \longrightarrow  0. \]

Let ${\mathcal V}=\{ V_i \}_i$ be an affine open cover of $X$, such
that every   sheaf of the complex  $\EE_\FF^{\cdot}$ is  free on  each $V_i$; then, consider the sheaf of dgLas $\EE nd^{\cdot}(\EE_\FF^{\cdot})$  and the associated semicosimplicial dgLa 
 \[ \EE nd^{\cdot}(\EE_\FF^{\cdot})(\mathcal V)\! :\!\! \xymatrixcolsep{6mm} \xymatrix{{\prod_i\EE nd^\cdot(\EE_\FF^\cdot)(V_i)}
	\ar@<2pt>[r]\ar@<-2pt>[r] & {
		\prod_{i,j}\EE nd^\cdot (\EE_\FF^\cdot) (V_{ij})}
	\ar@<4pt>[r] \ar[r] \ar@<-4pt>[r] &
	{\prod_{i,j,k}\EE nd^\cdot (\EE_\FF^\cdot)(V_{ijk})}
	\ar@<6pt>[r] \ar@<2pt>[r] \ar@<-2pt>[r] \ar@<-6pt>[r]& \cdots}.
\]

\begin{theorem}\label{thm.dgLa def fasci coerenti}

In the above notations, there exist  equivalences   of pseudo-functors
\[  \Del_{\Tot (\EE nd^{*}(\EE_\FF^{\cdot})(\mathcal V))}  \cong \Tot(\Del_{\EE nd^{*}(\EE_\FF^{\cdot})(\mathcal V)}) \cong \Def_{\FF}.\]

\end{theorem}
\begin{proof}
The first equivalence follows from Theorem~\ref{generalized hinich theorem}. Indeed, the negative ext-groups of a coherent sheaf are zero and so we can apply the descent of Deligne groupoids.

The second one    
is investigated in \cite[Section 2]{FIM} for functors in $\Set$. For the reader's convenience, we summarise here the main computations,  extending  them to the level of groupoids.  

We define an equivalence of pseudo-functors 
\[ \Theta: \Tot(\Del_{\EE nd^{*}(\EE_\FF^{\cdot})(\mathcal V)}) \to \Def_{\FF}\]
as follows.
Let $A \in \Art_\kk$, the functor
\[ \Theta(A): \Tot(\Del_{\EE nd^{*}(\EE_\FF^{\cdot})(\mathcal V)})(A) \to \Def_{\FF}(A)\]
associates to every object $(l,m) \in \Tot(\Del_{\EE nd^{*}(\EE_\FF^{\cdot})(\mathcal V)}) (A)$ the deformation $\FF_A$ of $\FF$ obtained in the following way. 
The first datum is a Maurer--Cartan element $l=\{l_i\}_i
\in \prod_i {\EE nd}^1(\EE_\FF^{\cdot})(V_i)\otimes \m_A$ defining, on every open set $V_i$, a deformation of the complex $(\EE_\FF^{\cdot}|_{V_i},d)$ given by
\[(\EE_\FF^{\cdot}|_{V_i}\otimes A,d + l_i).\]
 Note that the condition of being a complex
 corresponds to the Maurer--Cartan equation for $l_i$ and that, by upper semicontinuity, it is exact except possibly at
zero level.
The second datum is a class, up to irrelevant stabiliser, of an element $m=\{m_{ij}\}_{ij} \in \prod_{i,j} {\EE nd}^0(\EE_\FF^{\cdot})(V_{ij})\otimes \m_A$ which defines isomorphisms 
\[e^{m_{ij}}: (\EE_\FF^{\cdot}|_{V_{j}}\otimes A,d + l_j)|_{V_{ij}} \to
(\EE_\FF^{\cdot}|_{V_{i}}\otimes A,d + l_i)|_{V_{ij}},
\] 
between the deformed local complexes restricted to the double intersections. Note that the compatibility with the differentials is assured by the equation $   l_i|_{V_{ij}}=e^{m_{ij}}*l_j|_{V_{ij}}$ for all $i,j$ and, because of their form, these isomorphisms reduce to the identity on the residue field. 
Moreover, by definition, there exists an element $n=\{n_{ijk}\}_{ijk} \in \prod_{i,j,k} {\EE nd}^{-1}(\EE_\FF^{\cdot})(V_{ijk}) \otimes \m_A$, such that
\[ m_{jk} |_{V_{ijk}} \bullet - m_{ik}|_{V_{ijk}} \bullet m_{ij}|_{V_{ijk}} =
d_{\EE
nd^{*}(\EE_\FF^{\cdot})}n_{ijk}+[l_j |_{V_{ijk}} , n_{ijk}]= [d + l_j |_{V_{ijk}} , n_{ijk}].
\]
The last equality says that the isomorphisms $e^{m_{ij}}$ satisfy the cocycle condition up to homotopy. 
Thus, the local
$A$-flat sheaves of $\OO_X|_{V_i}\otimes A$-modules ${\mathcal
F}_{A,V_i}:={\mathcal H}^0(\EE_\FF^{\cdot}|_{V_i}\otimes A,d + l_i)$ 
glue together to give the global coherent sheaf ${\mathcal F}_A$
which is a deformation of ${\mathcal F}$. Indeed, we do not need to glue the deformed complexes together but only their cohomology.

Observe that the definition $\Theta(A)(l,m):=\FF_A$ does not depend on the class of $m$ up to irrelevant stabiliser.
Indeed, let $m'=\{ m'_{ij} \}_{ij}\in \prod_{i,j} {\EE nd}^0(\EE_\FF^{\cdot})(V_{ij})\otimes \m_A$ and $u=\{u_{ij}\}_{ij} \in  \prod_{i,j} {\EE nd}^{-1}(\EE_\FF^{\cdot})(V_{ij})\otimes \m_A$, such that
\[ m_{ij}= m'_{ij} \bullet \left(d_{ \EE nd^{*}(\EE_\FF^{\cdot})} u_{ij} + [l_j|_{V_{ij}}, u_{ij} ]\right)= m'_{ij} \bullet [d + l_j |_{V_{ij}} , u_{ij}].\]
The first equality expresses the fact that $m$ and $m'$ are in the same class up to irrelevant stabiliser, while the last assures that the local isomorphisms $e^{m_{ij}}$ and $e^{m'_{ij}}$ induce the same isomorphisms in cohomology. Since the global deformations $\Theta(A)(l,m)$ and $\Theta(A)(l,m')$ are obtained gluing the same local deformations with the same gluing functions, they are the same.

Next, we define the functor 
\[ \Theta(A): \Tot(\Del_{\EE nd^{*}(\EE_\FF^{\cdot})(\mathcal V)})(A) \to \Def_{\FF}(A)\]
on morphisms. 

A morphism in $ \Tot(\Del_{\EE nd^{*}(\EE_\FF^{\cdot})(\mathcal V)})(A)$ between the elements $(l,e^{m})$ and $(l',e^{m'})$ is a class, up to irrelevant stabiliser, of an element  $a=\{  a_i\}_i \in \prod_i {\EE nd}^0(\EE_\FF^{\cdot})(V_i) \otimes \m_A$. It defines, locally on every open set $V_i$, the isomorphisms 
\[ e^{a_i}: (\EE_\FF^\cdot|_{V_i}\otimes A, d+l_i) \to (\EE_\FF^\cdot|_{V_i}\otimes A, d+l'_i)\]
 between the deformed local complexes. As above, the compatibility with the differentials is assured by the equation $   l'_i=e^{a_{i}}*l_i$ and they reduce to the identity on the residue field. 
Moreover,
there exists  $b=\{b_{ij}\}_{i,j} \in \prod_{i,j} {\EE nd}^{-1}(\EE_\FF^{\cdot})(V_{ij})\otimes
\m_A$ that satisfies equations 
\[ -m_{ij} \bullet -a_i|_{V_{ij}} \bullet m'_{ij} \bullet a_j
|_{V_{ij}} = d_{\EE nd^{*}(\EE_\FF^{\cdot})}b_{ij}+ [l_j|_{V_{ij}}, b_{ij}].\]
This means that the local isomorphisms $e^{a_i}$ glue together in cohomology to give a global isomorphism $\Theta(A)(a)$ of the corresponding deformed sheaves $\Theta(A)(l,m)={\mathcal F}_A$ and $\Theta(A)(l',m')=\mathcal{F}_A'$.

Note that $\Theta(A)$ is well defined on the class of $a$, up to irrelevant stabiliser. Indeed, let  $a'=\{  a'_i\}_i \in \prod_i {\EE nd}^0(\EE_\FF^{\cdot})(V_i) \otimes \m_A$ and 
 $v=\{v_i\}_i \in  \prod_{i} {\EE nd}^{-1}(\EE_\FF^{\cdot})(V_{i})\otimes \m_A$, such that
\[ a_{i}= a'_{i} \bullet \left(d_{ \EE nd^{*}(\EE_\FF^{\cdot})} v_{i} + [l_i, v_{i} ]\right)= a'_i \bullet [d + l_i , v_{i}].\]
The first equality expresses the fact that $a$ and $a'$ represent the same class up to irrelevant stabiliser, while the last assures that the local isomorphisms $e^{a_i}$ and $e^{a'_i}$ induce the same isomorphisms in cohomology, i.e., between the deformed sheaves $\FF_A|_{V_i}$ and $\FF'_A|_{V_i}$, on every open set $V_i$. Thus, coinciding locally, the isomorphisms $\Theta(A)(a)$  and $\Theta(A)(a')$ coincide.

The facts that $\Theta(A)$ is an isomorphism at the level of $\pi_0$ and that it is full on morphisms follow by the above explicit description and all details are analysed in    \cite[Section 2]{FIM}.

Finally, we prove that $\Theta(A)$ is faithful on morphisms. Let  $a=\{  a_i\}_i,  a'=\{  a'_i\}_i  \in \prod_i {\EE nd}^0(\EE_\FF^{\cdot})(V_i) \otimes \m_A$ be morphisms of $\Tot(\Del_{\EE nd^{*}(\EE_\FF^{\cdot})(\mathcal V)})(A) $ between $(l,m)$ and $(l',m')$, 
such that $\Theta(A)(a)= \Theta(A)(a')$. As already discussed, $a$ and $a'$ define local isomorphisms between the local deformed complexes constructed using $l$ and $l'$. They do not need to be equal, the condition $\Theta(A)(a)= \Theta(A)(a')$ means that they 
 induce the same isomorphism in cohomology between the sheaves $\Theta(A)(l,m)=\FF_A$ and $\Theta(A)(l',m')=\FF'_A$. 
As above, this fact is equivalent to the fact that $a$ and $a'$ are in the same class up to irrelevant stabiliser. 
The last check is that $\Theta$ is a pseudo-natural transformation. 
Let $f: A \to B$ be a morphism in the category $\Art_\kk$. The pseudo-naturality lies in the fact that, given an object $(l,m) \in \Tot(\Del_{\EE nd^{*}(\EE_\FF^{\cdot})(\mathcal V)})(A) $,  
the two deformations $\HH^\cdot(\EE_\FF^{\cdot} \otimes A, d+l) \otimes_A B$ and $\HH^\cdot(\EE_\FF^\cdot \otimes A \otimes_A B, d+ f(l))$  are isomorphic via natural isomorphisms. 
\end{proof}

\begin{remark}
As a consequence, we recover the classical well known fact that the functor of
infinitesimal deformations of $\mathcal F$ has $\Ext_X^1({\mathcal
F},{\mathcal F})$ as tangent space and its obstructions are
contained in $\Ext_X^2({\mathcal F},{\mathcal F})$.
\end{remark}
\begin{remark} \label{rmk.indipendenza risoluzioni def fasci}
The above description of the functor of infinitesimal deformations of ${\mathcal F}$ is actually independent of the resolution chosen. Indeed, the dgLas of the
endomorphisms of any two locally free resolutions of ${\mathcal F}$ are
quasi-isomorphic (see,    \cite[Lemma 4.4]{Seidel.Thomas} or Lemma~\ref{lemma.indipendenza risoluzioni def sottofascio}). 
 Moreover, the computation does not depend on the choice of the affine cover \cite{FIM}.  
\end{remark}
 
\section{Deformations of a subsheaf inside a sheaf}\label{sec.subsheaf}

In this section, we analyse infinitesimal deformations of pairs consisting of a sheaf and a subsheaf of it.

Let $X$ be a smooth variety, $\FF$ and $\GG$ coherent sheaves of $\OO_X$-modules over $X$, such that $\FF$ is a subsheaf of $\GG$; denote by  $i: \FF \hookrightarrow \GG$  the natural inclusion.

\begin{definition}  
An \emph{infinitesimal deformation of the pair $\FF\subset\GG$ over $A\in \Art_{\kk}$} is given by a commutative diagram

\begin{center}
		\begin{tikzcd}
			 \FF_A  \arrow[dr]  \arrow[bend left, rr, "\pi_A^\GG|_{\FF_A}"]  \arrow[hook, r]  & \GG_A \ar[d]  \arrow[bend  left, rr, "\pi_A^\GG"]&    \FF \arrow[d]    \arrow[hook, r] &  \GG   \arrow[dl]  \\
 &\Spec A \arrow[r]  & \Spec \kk,  & 		\end{tikzcd}
	\end{center}
where $(\GG_A, \pi_A^\GG)$  and  $(\FF_A, \pi_A^\GG|_{\FF_A})$  are infinitesimal deformations over $A$  of the sheaves $\GG$  and $\FF$, respectively.

Two deformations $\mathcal{F}_A \hookrightarrow \GG_A$ and $\mathcal{F}'_A\hookrightarrow \GG'_A$ of $\mathcal{F}\hookrightarrow \GG$ over $A$ are isomorphic
if there exists an
isomorphism $\GG_A \to \GG'_A$  of the deformations of the sheaf $\GG$ that restricts to an isomorphism of deformations of sheaves $\FF_A \to \FF'_A$. 
\end{definition}

\begin{definition} 
The \emph{pseudo-functor of infinitesimal deformations of the pair $\FF\subset\GG$} is defined as
\[  \Def_{\FF\subset\GG} : \Art_\kk \to \Grpds,\] 
that associates, to any $A \in \Art_\kk$, the groupoid $ \Def_{\FF\subset\GG}(A)$, whose objects are the infinitesimal deformations 
of the pair $\FF\subset\GG$ over
 $A$ and whose morphisms are the isomorphisms of them.
  Note that  
\[ \pi_0(\Def_{\FF\subset\GG}(A)) =\]\[= \{\mbox{isomorphism classes of infinitesimal deformations of the pair $\FF\subset\GG$ over $A$}\} \]
is the classical functor of deformations in $\Set$   (see for example \cite[Section 2.A]{HuyLehn}). 

\end{definition}

Here, we would like to recall some  facts that  will be useful  later.

Let $\EE_\FF^{\cdot} \to \FF$   be a bounded locally free resolution of $\FF$   and $\EE_\GG^{\cdot} \to \GG$ be  a bounded locally free resolution of $\GG$, such that  $\EE_\FF^i$ is a subsheaf of $\EE_\GG^i$, for all $i\leq 0$, and the following diagram is commutative: 
\begin{equation} \label{equazione risoluzione sottofascio}
	\begin{tikzcd}
		0 \arrow[r] & \EE^{-n}_\FF \arrow[r] \arrow[d, hook] & \cdots \arrow[r] & \EE^{-1}_\FF \arrow[r] \arrow[d, hook] & \EE^{0}_\FF \arrow[r] \arrow[d, hook] & \FF \arrow[r] \arrow[d, "i", hook] & 0 \\
		0 \arrow[r] & \EE^{-n}_\GG \arrow[r]                 & \cdots \arrow[r] & \EE^{-1}_\GG \arrow[r]                 & \EE^{0}_\GG \arrow[r]                 & \GG \arrow[r]                      & 0,
	\end{tikzcd}
\end{equation}
where $n \leq \dim X$.
Note that, thanks to a version of the Horseshoe Lemma \cite[Section 2.2]{Weibel} for locally free resolutions,
it is always possible to construct two resolutions with such properties \cite[B.8.3]{Fulton}. 
We will shorten such a diagram as

\begin{center}
	\begin{tikzcd}
		\EE_\FF^\cdot \arrow[r, two heads] \arrow[d, hook] &  \FF \arrow[d, "i", hook] \\
		\EE_\GG^\cdot \arrow[r, two heads]                  & \GG .                     
	\end{tikzcd}
\end{center}

Next,  define the sheaf of dgLas
\begin{equation}\label{definizione L} 
\mathcal L^\cdot =\{ \varphi \in \EE nd^\cdot (\EE_\GG^\cdot) \mid \varphi(\EE_\FF^\cdot) \subset \EE_\FF^\cdot \}, 
\end{equation}
of the endomorphisms of the resolution $\EE_\GG^\cdot$ that restrict to endomorphisms of the subcomplex $\EE_\FF^\cdot$.

Actually, the quasi-isomorphism class of the sheaf of dgLas $\mathcal L^\cdot$ does not depend on the choice of the two resolutions. Indeed, mimicking \cite[Lemma 4.4]{Seidel.Thomas}, we can prove the following result.

\begin{lemma} \label{lemma.indipendenza risoluzioni def sottofascio}
Let $\EE_\FF^\cdot \hookrightarrow \EE_\GG^\cdot$  and $\EE_\FF'^\cdot \hookrightarrow \EE_\GG'^\cdot$
 be two bounded resolutions of $\FF  \hookrightarrow \GG$ that fit into a commutative diagram like \eqref{equazione risoluzione sottofascio}. 
Let $\mathcal L^\cdot$ be the sheaf of dgLas defined in \eqref{definizione L}  and 
\[ \mathcal {L'^\cdot}=\{ \varphi \in \EE nd^\cdot (\EE_\GG'^\cdot) \mid \varphi(\EE_\FF'^\cdot) \subset \EE_\FF'^\cdot \}. \]
 Then,  $\mathcal L^\cdot$ and $\mathcal L'^\cdot$ are quasi-isomorphic. 
\end{lemma}
\begin{proof}
Following \cite[Theorem 4.4]{joint-defos} and \cite[Lemma 7.7]{pairs}, via the usual killing cycles procedure,  it is possible to construct bounded complexes of locally free sheaves $\mathcal{P}^\cdot,\mathcal{N}^\cdot, \mathcal{Q}^\cdot$ and $\mathcal{R}^\cdot$ which fit into a commutative diagram of the form
	\begin{center}
		\begin{tikzcd}
			&                                                                  & \FF \arrow[ddd, dashed, hook, bend right, shift right=0.5]                        &                                             &   \\
			0 \arrow[r] & \EE_\FF^\cdot \oplus \EE_\FF'^\cdot \arrow[r] \arrow[d, hook] \arrow[ru] & \mathcal{Q}^\cdot \arrow[r] \arrow[d, hook] \arrow[u, two heads] & \mathcal{R}^\cdot \arrow[r] \arrow[d, hook] & 0 \\
			0 \arrow[r] &\EE_\GG^\cdot\oplus \EE_\GG'^\cdot \arrow[r] \arrow[rd]                  & \mathcal{P}^\cdot \arrow[r] \arrow[d, two heads]                 & \mathcal{N}^\cdot \arrow[r]                 & 0 \\
			&                                                                  & \GG,                                                              &                                             &  
		\end{tikzcd}
	\end{center}
	where the rows are short exact sequences, $\mathcal{Q}^\cdot \to \FF$ and $\mathcal{P}^\cdot \to \GG$ are locally free resolutions, and the induced morphisms $i_1 \colon \EE_\FF^\cdot \to \mathcal{Q}^\cdot$, $i_2 \colon \EE_\FF'^\cdot \to \mathcal{Q}^\cdot$, $j_1\colon\EE_\GG^\cdot \to \mathcal{P}^\cdot$, $j_2\colon\EE_\GG'^\cdot \to \mathcal{P}^\cdot$ are quasi-isomorphisms.  This construction is detailed in the appendix (see Lemma~\ref{lemma.somma-diretta-risoluz}).
	Consider the commutative  diagram 
	\begin{center}
		\begin{tikzcd}
			\EE_\FF^\cdot \arrow[r, "i_1", "\sim"',hook] \arrow[d, hook] & \mathcal{Q}^\cdot \arrow[d, hook] \\
			\EE_\GG^\cdot \arrow[r, "j_1","\sim"', hook]                & \mathcal{P}^\cdot,               
		\end{tikzcd}
	\end{center}
 	where $j_1$ is a quasi-isomorphism between two resolutions of $\GG$ such that $j_1(\EE_\FF^\cdot) \subset \mathcal{Q}^\cdot$.
	Consider the cone $\CC^\cdot=\operatorname{cone} (j_1)$ and 
	the sheaf of dgLas  of the endomorphisms of the cone $\CC^\cdot$:
	\[ \EE nd^\cdot (\CC^\cdot) = \left\{ \phi=\left( \begin{array}{cc}  \phi_{11} & \phi_{12} \\ \phi_{21} & \phi_{22}\end{array}  \right)  \bigg\vert \begin{array}{cc}
		 \phi_{11} \in \EE nd^\cdot(\EE_\GG^\cdot), &\phi_{12} \in \HH om ^{\cdot+1}(\mathcal{P}^\cdot, \EE_\GG^\cdot),\\ \phi_{21} \in \HH om^{\cdot-1}(\EE_\GG^\cdot, \mathcal{P}^\cdot), &\phi_{22} \in \EE nd^\cdot(\mathcal{P}^\cdot)
	\end{array}\right\}\!,\] 
	whose differential $d: \EE nd^k (\CC^\cdot)  \to \EE nd^{k+1} (\CC^\cdot)  $ is given by: 
	\[ d\phi=\!\!\left( \begin{array}{cc}  - d_{\EE^\cdot_\GG} \phi_{11} +(-1)^k \phi_{11} d_{\EE^\cdot_\GG}  -(-1)^k \phi_{12} j_1 & -d_{\EE^\cdot_\GG}  \phi_{12}-(-1)^k \phi_{12} d_{\mathcal{P}^\cdot}  \\
		j_1 \phi_{11} -(-1)^k \phi_{22} j_1 + d_{\mathcal{P}^\cdot} \phi_{21} + (-1)^k \phi_{21}d_{\EE^\cdot_\GG}  & j_1 \phi_{12} + d_{\mathcal{P}^\cdot}  \phi_{22} -(-1)^k \phi_{22} d_{\mathcal{P}^\cdot}  \end{array}\right)\!. \]
Next, we define the following subsheaf of $\EE nd^\cdot (\CC^\cdot) $:  
\[ \mathcal D^\cdot\!=\! \left\{\phi= \left( \begin{array}{cc}  \phi_{11} & 0 \\ \phi_{21} & \phi_{22}\end{array}  \right) \in \EE nd^\cdot (\CC^\cdot) \ \bigg\vert \ \phi_{11} (\mathcal{E}_\FF^\cdot) \subset \mathcal{E}_\FF^\cdot, \  \phi_{21} (\mathcal{E}_\FF^\cdot) \subset \mathcal{Q}^\cdot,  \  \phi_{22} (\mathcal{Q}^\cdot) \subset \mathcal{Q}^\cdot  \right\} \!; \]
observe that it is closed with respect to the above differential, because $j_1(\EE_\FF^\cdot) \subset  \mathcal{Q}^\cdot$ and then it is still a sheaf of dgLas. 

The projection $\pi_2:  \mathcal{D}^\cdot \to \mathcal{M}^\cdot := \{ \varphi \in \EE nd^\cdot (\mathcal{P}^\cdot) \mid \varphi(\mathcal{Q}^\cdot) \subset \mathcal{Q}^\cdot \}$, given by $\phi \mapsto \phi_{22}$, is  surjective, and, up to a shift, its kernel is isomorphic to the cone of ${(j_1)}_* \colon \LL^\cdot \to \{ \varphi \in \HH om^\cdot (\EE_\GG^\cdot,\mathcal{P}^\cdot) \mid \varphi(\EE_\FF^\cdot) \subset \mathcal{Q}^\cdot\}$. Since $j_1$ is a quasi-isomorphism between  bounded  locally free resolutions,  it is a quasi-isomorphism between bounded complexes of projective modules on every affine open set; then it is a homotopy equivalence \cite[III.5.24]{GM}. Therefore, on every affine open subset of $X$, ${(j_1)}_*$ is a quasi-isomorphism, hence its cone is acyclic.    
Since the map $\pi_2$ is surjective,  the kernel of $\pi_2$ is quasi-isomorphic to the cone of $\pi_2$ up to a shift; then it is acyclic  and so  $\pi_2$ is a quasi-isomorphism.   
The same argument can be used for the map $\pi_1:  \mathcal{D}^\cdot \to \LL^\cdot$, so that we get a zigzag of quasi-isomorphisms of sheaves of dgLas
\begin{center}
	\begin{tikzcd}
		& \mathcal{D}^\cdot \arrow[ld, "\pi_1"', two heads] \arrow[rd, "\pi_2", two heads] &             \\
		\LL^\cdot &                                                                            & \mathcal{M}^\cdot.
	\end{tikzcd}
\end{center}
Applying the same construction to the diagram of resolutions
\begin{center}
	\begin{tikzcd}
		\EE_\FF'^\cdot \arrow[r, "i_2", "\sim"',hook] \arrow[d, hook] & \mathcal{Q}^\cdot \arrow[d, hook] \\
		\EE_\GG'^\cdot \arrow[r, "j_2","\sim"', hook]                & \mathcal{P}^\cdot,               
	\end{tikzcd}
\end{center}
we obtain another zigzag of quasi-isomorphisms, so that we have
\begin{center}
	\begin{tikzcd}
		& \mathcal{D}^\cdot \arrow[ld, "\pi_1"', two heads] \arrow[rd, "\pi_2", two heads] &             & \mathcal{D}'^\cdot \arrow[ld, two heads] \arrow[rd, two heads] &      \\
		\LL^\cdot &                                                                            & \mathcal{M}^\cdot &                                                          & \LL'^\cdot,
	\end{tikzcd}
\end{center}
and then $\LL^\cdot$ and $\LL'^\cdot$ are quasi-isomorphic.
\end{proof}

\begin{remark} \label{rmk.coomologia non negativa per L}
 Let $\iota : \EE_\FF^\cdot \hookrightarrow \EE_\GG^\cdot$ be a resolution of $ \FF\hookrightarrow  \GG$ such that the $\coker (\iota)$ is a locally free resolution of $\GG/\FF$; it is always possible to construct such resolutions (see \cite[B.8.3]{Fulton}): 
 \[ 0 \to \EE_\FF^\cdot \hookrightarrow \EE_\GG^\cdot \to \coker(\iota) \to 0.\]
Let $\LL^\cdot$ be as defined above, then we have the short exact sequence: 
 \[ 0 \to \mathcal{L}^\cdot \to \EE nd^\cdot (\EE_\GG^\cdot)  \to  \HH om^\cdot (\EE_\FF^\cdot,\coker(\iota)) \to 0 \]
and the induced long exact sequence in hypercohomology
\[  \cdots\to \mathbb{H}^i(X ,\mathcal{L}^\cdot) \to   \mathbb{H}^i(X , \EE nd^\cdot(\EE^\cdot_\GG) )\to \mathbb{H}^i(X ,\HH om^\cdot (\EE_\FF^\cdot,\coker(\iota))) \to \cdots .\]
Note that $\mathbb{H}^i(X , \EE nd^\cdot (\EE_\GG^\cdot) ) \cong \Ext_X^i(\GG,\GG)$ and   $\mathbb{H}^i(X,\HH om^\cdot (\EE_\FF^\cdot,\coker(\iota)) )\cong \Ext_X^i(\FF, \GG/\FF)$, so they all vanish for all $i<0$ and the same holds for $\mathcal{L}^\cdot$.
\end{remark}

Finally, we find the  dgLa that controls the deformations $\Def_{\FF\subset\GG}$.

\begin{theorem}\label{thm.dgLa def sottofasci}

In the above notations,  there are equivalences of functors
\[  \Del_{\Tot (\mathcal L^\cdot(\mathcal V))} \cong   \Tot(\Del _{  (\mathcal L^\cdot(\mathcal V))})\cong \Def_{\FF\subset \GG}.\]
\end{theorem}

\begin{proof}  
The first equivalence is a consequence of Theorem~\ref{generalized hinich theorem}: note that, by Remark~\ref{rmk.coomologia non negativa per L}, the hypotheses are satisfied.

To prove the second equivalence it is enough to follow  step by step the approach described in the proof of Theorem~\ref{thm.dgLa def fasci coerenti} (see also \cite[Section 2]{FIM}).
 Here we have just to be careful that the deformations of the pair $\FF\subset \GG$ have to be compatible with the inclusion. This fact is assured since the endomorphisms of $\EE^\cdot_\GG$ we are considering are contained in $\LL^\cdot$.  
\end{proof}

\begin{remark}\label{rem.forgetful-functor}
As before, the above  computation does not depend on the choice of the affine cover \cite{FIM}. 
Note that we can define a forgetful  functor    $\Def_{\FF\subset\GG} \to \Def_\GG$, forgetting all information  related to   $\FF$. According to  Theorem~\ref{thm.dgLa def sottofasci},  $\Def_{\FF\subset\GG} $ is controlled by the dgLa associated to $
 \mathcal L^\cdot$ while, by Theorem~\ref {thm.dgLa def fasci coerenti},  $\Def_\GG$ is controlled  by the dgLa associated to $ \EE nd^\cdot (\EE_\GG^\cdot) $.
 Then, the forgetful  functor corresponds to the morphism $ \mathcal L^\cdot \to  \EE nd^\cdot (\EE_\GG^\cdot)$.
For example, thanks to Remark~\ref{rmk.coomologia non negativa per L}, if $\mathbb{H}^1(X,\HH om^\cdot (\EE_\FF^\cdot,\coker(\iota))) \cong \Ext_X^1(\FF, \GG/\FF)=0$, then the forgetful functor  $\Def_{\FF\subset\GG} \to \Def_\GG$ is formally smooth according to \cite[Definition 6.22]{TV}. 
\end{remark}

\section{Deformations of a morphism of coherent sheaves} \label{sezione deformazione mappe}

In this section, we analyse the infinitesimal deformations of a morphism of coherent sheaves and explicitly describe a differential graded Lie algebra that controls this deformation problem.

\subsection{Geometric deformations}
Let $X$ be a smooth variety over a field $\mathbb{K}$ of characteristic zero,   $\FF$ and $\GG$  coherent sheaves of $\OO_X$-modules over $X$ and  $\alpha \colon \FF\to \GG$  a morphism of them. 
First of all, we recall some definitions. 
\begin{definition} \label{def.deform di morf di fasci}
	\emph{An infinitesimal deformation of the  morphism  $\alpha \colon \FF\to \GG$ over  $A\in \Art_\kk$} is a morphism $\alpha_A: \FF_A\to \GG_A$ of coherent sheaves of $\OO_X \otimes A$-modules over $X\times \Spec A$, where $\FF_A$ and $\GG_A$ are deformations of $\FF$ and $\GG$ over $A$ respectively, such that the following diagram is commutative: 
	\[
	\xymatrix{ \FF_A \ar[r]^{\alpha_A} \ar[d]^{\pi_A^\FF}  & \GG_A \ar[r] \ar[d]^{\pi_A^\GG} & \Spec A  \ar[d] \\
		\FF \ar[r]^{\alpha} & \GG \ar[r] & \Spec \kk .  }
	\]

	Two deformations $\alpha_A: \FF_A\to \GG_A$ and $\alpha'_A: \FF'_A\to \GG'_A$  of $\alpha: \FF \to \GG$ over $A$ are \emph{isomorphic} 
	if there exists a pair $(\phi,\psi)$ of isomorphisms of sheaves $\phi: \FF_A \to \FF'_A$ and $\psi : \GG_A \to \GG'_A$, such that the following diagram commutes:
	\[
	\xymatrix{ \FF_A \ar[rrr]^{{\alpha}_A} \ar[dd]^\phi \ar[dr]^{\pi_A^{\FF}}  & & & \GG_A \ar[dl]_{\pi_A^{\GG}}   \ar[dd]^\psi \ar[dr] & \\
		& \FF \ar[r]^\alpha & \GG  & & \Spec A \\
		\FF'_A \ar[rrr]_{\alpha'_A } \ar[ur]_{\pi_A^{\FF'}}   & & &     \GG'_A \ar[ur] \ar[ul]^{\pi_A^{\GG'}} . &  }\]
	
\end{definition}

We recall that the trivial deformation of a sheaf $\FF$ over $A$ is given by  $\FF\otimes_\kk A $ and that the trivial deformation of $\alpha \colon \FF\to \GG$   over  $A\in \Art_\kk$ is given by the trivial extension 
$ \alpha\otimes \Id_A \colon \FF\otimes_\kk A  \to \GG\otimes_\kk A $.

\begin{definition} \label{def.funtore def in gruppoidi di alpha}
The \emph{pseudo-functor of infinitesimal deformations of  the morphism $\alpha:\FF \to \GG$} is defined as
\[  \Def_{(\FF,\alpha, \GG)}  : \Art_\kk \to \Grpds,\] 
	that associates, to any $A \in \Art_\kk$, the groupoid $\Def_{(\FF,\alpha, \GG)}(A)$, whose objects are the infinitesimal deformations of the morphism $\alpha$  over
	$A$ and whose morphisms are the isomorphisms of them.
	Note that  
	\[ \pi_0(\Def_{(\FF,\alpha, \GG)}(A)) =\]
	\[= \{\mbox{isomorphism classes of infinitesimal deformations of the morphism $\alpha$ over $A$}\} \]
	is the classical functor of deformations in $\Set$.
	
\end{definition} 

\begin{remark}\label{remark defo di morfismi come grafo}
	Among all the  infinitesimal deformations of $\alpha \colon \FF\to \GG$ over $A \in\Art_\kk$, there are the infinitesimal deformations of the morphism  $\alpha$ in which $\FF$ and $\GG$ deform trivially, i.e., the deformations $ \alpha_A \colon \FF\otimes_\kk A  \to \GG\otimes_\kk A $ such that just the map deforms.  The  groupoid of these deformations defines a sub-pseudo-functor 
	$\Def_{\alpha} : \Art_\kk \to \Grpds$. 
of the pseudo-functor $\Def_{(\FF,\alpha, \GG)} $. With \emph{sub-pseudo-functor} we intend that for every    $A \in \Art_\kk$ the functor $\Def_{\alpha}(A) \to \Def_{(\FF,\alpha, \GG)}(A) $ is injective on $\pi_0$ and fully faithful.  
\end{remark}

Let  $\alpha \colon \FF\to \GG$  be a morphism of coherent sheaves and $\gamma$ its graph in
$ \FF\oplus \GG$. We define $\gamma$ as the image of the morphism of sheaves $(\Id, \alpha): \FF \to \FF\oplus \GG$. 

For future use, we would like to observe  that a deformation of $\alpha: \FF \to \GG$ over $A \in \Art_\kk$,    as  in Definition~\ref{def.deform di morf di fasci}, can be also  viewed  as the collection of the following data: 

\begin{itemize}
	\item two deformations $\FF_A$ and $\GG_A$ of $\FF$ and $\GG$ over $A$, respectively; 
	\item a deformation $\gamma_A \subset (\FF\oplus\GG)_A$ of the pair   $\gamma \subset \FF\oplus\GG$  over $A$; 
	\item an isomorphism $f_A$ between the deformations  $\FF_A \oplus \GG_A$ and $(\FF\oplus\GG)_A$. 
\end{itemize}
Two such collections of data $\left(\FF_A, \GG_A, \gamma_A \subset (\FF\oplus\GG)_A, f_A \right)$ and $\left(\FF'_A, \GG'_A, \gamma'_A \subset (\FF\oplus\GG)'_A, f'_A \right)$ define isomorphic deformations of $\alpha:\FF \to \GG$, if there exist:
\begin{itemize}
	\item two isomorphisms $\phi: \FF_A\to \FF'_A$ and $\psi:\GG_A\to \GG'_A$ of the deformations of $\FF$ and $\GG$ over $A$, respectively, 
	\item  
	an isomorphism $\chi:(\FF\oplus\GG)_A\to (\FF\oplus\GG)'_A$ of the deformations of $\FF\oplus \GG$ over $A$, with $\chi(\gamma_A)  \subset \gamma'_A$,
\end{itemize}
such that the following diagram is commutative
\[ \xymatrix{ \FF_A \oplus \GG_A \ar[r]^{f_A} \ar[d]^{(\phi,\psi)}  & (\FF\oplus\GG)_A \ar[d]^\chi \\
	\FF'_A \oplus \GG'_A \ar[r]^{f'_A}  & (\FF\oplus\GG)'_A .
}  \]

Remembering the definition of the total groupoid given in Definition~\ref{def. total grpds}  the above description can be summarised in the following result.

\begin{proposition} \label{prop.equiv grpds}
	The pseudo-functor $\Def_{(\FF, \alpha, \GG)}$ of infinitesimal deformations of a morphism $\alpha: \FF \to \GG$ of coherent sheaves is equivalent to the pseudo-functor that associates to every $A\in \Art_\kk$ the total groupoid of the following semicosimplicial groupoid: 
	\[ \xymatrix{  \left(\Def_{\FF}(A) \oplus \Def_{\GG } (A)\right) \oplus  \Def_{\gamma \subset \FF \oplus \GG}(A) 
		\ar@<2pt>[r]\ar@<-2pt>[r] &  \Def_{\FF\oplus \GG}(A) 
		\ar@<4pt>[r] \ar[r] \ar@<-4pt>[r] & 0}.\]
The first   arrow 
	\[\Def_{\FF}(A) \oplus \Def_{\GG } (A)  \to  \Def_{\FF\oplus \GG}(A),\]
	 associates to every pair $(\FF_A, \GG_A)$ of infinitesimal deformations of $\FF$ and $\GG$ over $A$ the infinitesimal deformation of $\FF \oplus \GG$ over $A$ given by the direct sum $\FF_A \oplus \GG_A$, and similarly on morphisms.
The second arrow is the forgetful functor $ \Def_{\gamma \subset \FF \oplus \GG}(A)  \to  \Def_{\FF\oplus \GG}(A)$ of Remark~\ref{rem.forgetful-functor}.
\end{proposition}

 This approach is similar to the one used in \cite{HorikawaI, HorikawaIII, Sernesi} to investigate deformations of holomorphic maps of complex manifolds and,  in  \cite{Dona.Tesi, Dona.def maps} and in \cite{Manetti.LMDT}, it has been applied to analyse these deformations via dgLas. 
More recently, this technique was used in \cite{DE maps} to study deformations of morphisms of locally free sheaves. Here, we are generalising this approach to investigate  morphisms of coherent sheaves. 

\subsection{The dgLa approach}
In this section, we describe a dgLa that controls the infinitesimal deformations of a morphism of coherent sheaves. 
In the notation above, 
let ${\mathcal V} = \{  V_i \}_{i}$ be an affine open cover of $X$ and 
let   $\EE_\FF^{\cdot} \to \FF$  and $\EE_\GG^{\cdot} \to \GG$   be bounded locally free resolutions such that there exists a lifting  $\alpha^\cdot: \EE_\FF^\cdot \to \EE_\GG^\cdot$ of the morphism $\alpha: \FF \to \GG$. It is always possible to construct such resolutions, see Lemma~\ref{lem.sollev-morf}. Let $\gamma^\cdot \subset \EE_\FF^\cdot \oplus \EE_\GG^\cdot$ be the graph of the lifting $\alpha^\cdot: \EE_\FF^\cdot \to\EE_\GG^\cdot$, which is a bounded locally free resolution of $\gamma$ by Corollary~\ref{cor.grafico}.
Consider $\EE_\FF^{\cdot} \oplus \EE_\GG^{\cdot} \to \FF \oplus \GG$ as a bounded  locally free resolution of $\FF \oplus \GG$. 
By Theorem~\ref{thm.dgLa def fasci coerenti}, there are the following equivalences:
\[
\Def_{\FF}  \cong \Del_{\Tot(\EE nd^\cdot(\EE_\FF^\cdot)({\mathcal V} ))}, \quad \Def_{\GG}  \cong \Del_{\Tot(\EE nd^\cdot(\EE_\GG^\cdot) ({\mathcal V} ))}, \]
\[ \Def_{\FF \oplus \GG } \cong \Del_{\Tot(\EE nd^\cdot(\EE_\FF^\cdot\oplus \EE_\GG^\cdot)({\mathcal V} ))}. \]
Note that by Remark~\ref{rmk.indipendenza risoluzioni def fasci}, 
the above descriptions of the pseudo-functors of infinitesimal deformations are actually independent of the  choices of the resolutions and affine cover.

Let 
\begin{equation}
	\mathcal L^\cdot =\{ \varphi \in \EE nd^\cdot (\EE_\FF^\cdot \oplus \EE_\GG^\cdot) \mid \varphi(\gamma^\cdot) \subset \gamma^\cdot \}
\end{equation}
and consider the following semicosimplicial dgLa 
\[ \mathcal L^\cdot(\mathcal V):  \xymatrix{ {\prod_i \mathcal L^\cdot(V_i)}
	\ar@<2pt>[r]\ar@<-2pt>[r] & {
		\prod_{i,j}\mathcal L^\cdot(V_{ij})}
	\ar@<4pt>[r] \ar[r] \ar@<-4pt>[r] &
	{\prod_{i,j,k}\mathcal L^\cdot(V_{ijk})}
	\ar@<6pt>[r] \ar@<2pt>[r] \ar@<-2pt>[r] \ar@<-6pt>[r]& \cdots}.
\]
By Theorem~\ref{thm.dgLa def sottofasci}, the following pseudo-functors are equivalent
\[ \Def_{\gamma\subset \FF \oplus \GG}  \cong \Del_{\Tot(\mathcal L^\cdot(\mathcal V))}.\]

Note that by Lemma~\ref{lemma.indipendenza risoluzioni def sottofascio}, 
	the above equivalence is actually independent of the choices of the resolutions of $\FF$ and $\GG$ and of the lifting of $\alpha$ and, as before, of the choice of the cover \cite{FIM}.

Consider the pseudo-functor $\Def_{(\FF, \alpha, \GG)}$, then for any $A \in \Art_\kk$, the equivalence of groupoids of Proposition~\ref{prop.equiv grpds} becomes
\[ \Def_{(\FF, \alpha, \GG)}(A)\cong \Tot \! \left( \!  \! \xymatrix{  \left(\Def_{\FF}(A) \oplus \Def_{\GG } (A)\right) \oplus  \Def_{\gamma \subset \FF \oplus \GG}(A) 
	\ar@<2pt>[r]\ar@<-2pt>[r] &  \Def_{\FF\oplus \GG}(A) 
	\ar@<4pt>[r] \ar[r] \ar@<-4pt>[r] & 0}  \!  \! \right)\cong \]
 {\small
\[  \!  \Tot \! \left( \!\! \!\!\xymatrix{  \left(  \Del_{\Tot(\EE nd^\cdot(\EE_\FF^\cdot)({\mathcal V} ))} (A)   \oplus \Del_{\Tot(\EE nd^\cdot(\EE_\GG^\cdot)({\mathcal V} ))} (A) \!  \right) \! \oplus \!   \Del_{\Tot(\LL^\cdot({\mathcal V} ))}(A) \ar@<2pt>[r]\ar@<-2pt>[r] &  \Del_{\Tot(\EE nd^\cdot(\EE_\FF^\cdot\oplus \EE_\GG^\cdot)({\mathcal V} ))} (A)} \!\!\!\!\right),  \]}
 
 Finally, applying the theorem on descent of Deligne groupoids (Theorem~\ref{generalized hinich theorem}), we obtain the following  equivalence
\[  \Def_{(\FF, \alpha, \GG)}\cong \]
{\small \[ \cong  \Del_{\Tot \left( \xymatrix{  \left( {\Tot(\EE nd^\cdot(\EE_\FF^\cdot)({\mathcal V} ))}  \oplus {\Tot(\EE nd^\cdot(\EE_\GG^\cdot)({\mathcal V} ))}  \right) \oplus   {\Tot(\LL^\cdot({\mathcal V} ))}\ar@<2pt>[r]\ar@<-2pt>[r] &  {\Tot(\EE nd^\cdot(\EE_\FF^\cdot\oplus \EE_\GG^\cdot)({\mathcal V} ))} }\right),  }  \]}
that can be summarised in our main result of the section.

\begin{theorem} \label{thm dgla of alpha}
	The pseudo-functor $\Def_{(\FF, \alpha, \GG)}$ of infinitesimal deformations of the morphism $\alpha:\FF \to \GG$ of coherent sheaves is equivalent to the Deligne functor associated to the Thom--Whitney dgLa $\Tot(H(\VV))$ of the following semicosimplicial dgLa $H(\VV)$:
	\[  \xymatrix{  \left( {\Tot(\EE nd^\cdot(\EE_\FF^\cdot)({\mathcal V} ))}  \oplus {\Tot(\EE nd^\cdot(\EE_\GG^\cdot)({\mathcal V} ))}  \right) \oplus   {\Tot(\LL^\cdot({\mathcal V} ))}\ar@<2pt>[r]\ar@<-2pt>[r] &  {\Tot(\EE nd^\cdot(\EE_\FF^\cdot\oplus \EE_\GG^\cdot)({\mathcal V} ))} }.\]
	In particular, $H^1(\Tot(H(\VV)))$ is the tangent space of $\Def_{(\FF, \alpha, \GG)}$ and  $H^2(\Tot (H(\VV)))$ is an obstruction space.
\end{theorem}

\begin{remark}
Throughout the paper, we assume that $X$ is a smooth, integral (and so irreducible and reduced), separated scheme of finite type over $\kk$. These conditions guarantee the existence of a bounded locally free resolution for any coherent sheaf and this allows us to construct compatible resolutions (see Section~\ref{sezione appendice risoluzioni}). Then, the previous theorem applies also in a more general case, whenever there exist such compatible resolutions.
\end{remark}

\begin{corollary}\label{cor.succ-esatta-lunga-H}
 The cohomology of the Thom--Whitney dgLa $\Tot(H(\VV))$ fits into the following exact sequence: 
\begin{equation}
 \cdots H^i(\Tot(H(\VV))) \to \Ext_X^i( \FF, \FF)  \oplus \Ext_X^i(\GG, \GG)  \stackrel{(-\alpha_*, \alpha^*)}{\longrightarrow} \Ext_X^i(\FF,\GG)     \to H^{i+1}(\Tot(H(\VV)) ) \cdots.
\end{equation}
where $\alpha_* \colon \End(\FF)\to \Hom(\FF, \GG)$ and $\alpha^* \colon \End(\GG)\to \Hom(\FF, \GG)$.
\end{corollary}

\begin{proof}
According to \cite[Section 4]{Dona.def maps}, for all $i\geq 0$, we have  the following exact sequence: 
\[  \ldots \to H^i(\Tot(H(\VV)) )\to \mathbb{H}^i( X, \EE nd^\cdot(\EE_\FF^\cdot) \oplus \EE nd^\cdot(\EE_\GG^\cdot) \oplus \LL^\cdot) \to \]
\[\to \mathbb{H}^i(X,\EE nd^\cdot(\EE_\FF^\cdot\oplus \EE_\GG^\cdot)) \to H^{i+1}(\Tot(H(\VV)) )\to  \ldots, \]
or equivalently, 
\[  \ldots \to H^i(\Tot(H(\VV)) )\to \Ext_X^i( \FF, \FF)  \oplus \Ext_X^i(\GG, \GG) \oplus \mathbb{H}^i( X, \LL^\cdot) \to \]
\[\to \Ext_X^i( \FF\oplus \GG,  \FF\oplus \GG) \to H^{i+1}(\Tot(H(\VV))) \to  \ldots, \]
Moreover,
since the map  $h:  \LL^\cdot \to    \EE nd^\cdot(\EE_\FF^\cdot\oplus \EE_\GG^\cdot)$ is injective, 
thanks to \cite[Lemma 3.1]{Dona.def maps},  the following sequence is also exact:
\begin{equation} \label{succ ex lunga generale}
 \cdots H^i(\Tot(H(\VV)) ) \to \Ext_X^i( \FF, \FF)  \oplus \Ext_X^i(\GG, \GG)  \to \mathbb{H}^i(X,\coker (h) )   \to H^{i+1}(\Tot(H(\VV)))  \cdots .
\end{equation}
From the constructions of the locally free resolutions, we have that $\coker (\kappa\colon \gamma^\cdot \to \EE_\FF^\cdot \oplus \EE_\GG^\cdot)$ is locally free  and $\gamma^\cdot \subset \EE_\FF^\cdot \oplus \EE_\GG^\cdot$ is the graph of the lifting $\alpha^\cdot: \EE_\FF^\cdot \to\EE_\GG^\cdot$. Then, by Remark~\ref{rmk.coomologia non negativa per L}, $\coker (h) \cong \HH om^\cdot (\gamma^\cdot, \coker (\kappa))$, therefore $\mathbb{H}^i(\coker (h) )\cong \Ext_X^i(\gamma, {(\FF\oplus \GG)}/\gamma)\cong \Ext_X^i(\FF,\GG)$. The long exact sequence \eqref{succ ex lunga generale} becomes the one required.

 Note that the map $\Ext_X^i( \FF, \FF)  \oplus \Ext_X^i(\GG, \GG)  \to  \Ext_X^i(\FF,\GG)   $ is induced by morphism of complexes of sheaves $\HH om^\cdot (\EE_\FF^\cdot, \EE_\FF^\cdot) \oplus \HH om^\cdot (\EE_\GG^\cdot, \EE_\GG^\cdot) \to  \HH om^\cdot (\gamma^\cdot, \coker (\kappa)) \cong \HH om^\cdot (\EE_\FF^\cdot , \EE_\GG^\cdot)$ sending $(f,g) \in \HH om^\cdot (\EE_\FF^\cdot, \EE_\FF^\cdot) \oplus \HH om^\cdot (\EE_\GG^\cdot, \EE_\GG^\cdot)$ to $g \alpha^\cdot - \alpha^\cdot f \in \HH om^\cdot (\EE_\FF^\cdot , \EE_\GG^\cdot)$, which we will denote by $(-\alpha^\cdot_*, {\alpha^{\cdot}}^*)$. In particular, it induces the map $(-\alpha_*, \alpha^*) \colon  \End(\FF)\oplus \End(\GG) \to \Hom(\FF, \GG)$, where we have used that $\alpha^\cdot$ is a lifting of $\alpha \colon \FF \to \GG$. 
\end{proof}

Observe that, because of the vanishing of negative ext groups for coherent sheaves, we obtain:  
\begin{equation}\label{successione esatta lunga H}
\begin{split} 
0 &\to	 {H^0(\Tot(H(\VV)))} \to  {\End(\FF)\oplus \End(\GG)} \stackrel{(-\alpha_*, {\alpha}^*)}  {\longrightarrow} \ {\Hom(\FF, \GG)      }   \to  \\
 \quad 	&\to {H^1(\Tot(H(\VV)))} \to  {\Ext_X^1(\FF, \FF)\oplus\Ext_X^1(\GG, \GG)}  \stackrel{(-\alpha_*, {\alpha}^*)} {\longrightarrow} {\Ext_X^1(\FF, \GG)} \to   \\
\quad &\to  {H^2(\Tot(H(\VV)))} \to   {\Ext_X^2(\FF, \FF)\oplus\Ext_X^2(\GG, \GG)}\ \stackrel{(-\alpha_*, {\alpha}^*)} {\longrightarrow} \   {\Ext_X^2(\FF, \GG)} \to   \cdots,
\end{split}
\end{equation}
so that $H^0(\Tot(H(\VV)))$ is the kernel of $(-\alpha_*, \alpha^*) \colon  \End(\FF)\oplus \End(\GG) \to \Hom(\FF, \GG)$, or equivalently, the fibre product of the maps $\alpha_* \colon \End(\FF)\to \Hom(\FF, \GG)$ and $\alpha^* \colon \End(\GG)\to \Hom(\FF, \GG)$. The map $(-\alpha_*, \alpha^*)$ is not trivial unless $\alpha=0$, because $\alpha \in \Hom(\FF, \GG)$ always belongs to its image.

On extension classes $ (e, e') \in {\Ext_X^1(\FF, \FF)\oplus\Ext_X^1(\GG, \GG)}$,  $(-\alpha_*, \alpha^*) \colon \Ext_X^1(\FF, \FF)\oplus\Ext_X^1(\GG, \GG) \to \Ext_X^1(\FF, \GG)$ is defined as the Baer sum of the pushout extension $-\alpha_* e$ and the pullback extension $\alpha^* e'$.

We can say more in the following specific cases. 

\begin{corollary}
If both $\FF$ and $\GG$ are rigid, i.e., $\Ext_X^1(\FF, \FF)=\Ext_X^1(\GG, \GG)=0$, the tangent space $H^1(\Tot(H(\VV)))$ is isomorphic to $\coker (-\alpha_*, \alpha^*)$. 
If $\FF$ and $\GG$ are simple, i.e., $\End(\FF)=\End(\GG)\cong\kk$, then also $H^0(\Tot(H(\VV)))\cong \kk$.
\end{corollary}

\begin{remark}
The pseudo-functor of deformations of the graph $\gamma$ of $\alpha$ as a subsheaf of $\FF \oplus \GG$, where $\FF \oplus \GG$ deforms trivially with trivial automorphisms, see \cite[13.4]{Manetti.LMDT}, has tangent space 
$\Hom(\FF, \GG)$ and obstruction space $\Ext_X^1(\FF, \GG)$. The snake maps in  diagram  \eqref{successione esatta lunga H} should be therefore thought as the tangent and obstruction maps  of this pseudo-functor to those of $\Def_{(\FF,\alpha, \GG)} $. 
\end{remark}

\begin{remark} 
If we work over the field of complex numbers, we can consider the  Dolbeault resolutions of the endomorphisms of sheaves, and avoid the \v{C}ech semicosimplicial dgLas. More precisely, we can consider the two morphisms of dgLas
  \[  
\xymatrix{ {  A_X^{0,*}( \EE nd^\cdot(\EE_\FF^\cdot) ) \oplus A_X^{0,*}(\EE nd^\cdot(\EE_\GG^\cdot)) \oplus A_X^{0,*}(\LL^\cdot)}
\ar@<2pt>[r] \ar@<-2pt>[r] &  A_X^{0,*}(\EE nd^\cdot(\EE_\FF^\cdot\oplus \EE_\GG^\cdot)).}
\]  
Then, the associated Thom--Whitney  dgLa  controls the infinitesimal deformations of $\alpha: \FF \to \GG$.  
\end{remark}

\section{Appendix}
 
For readers' convenience, we  collect  here some useful computations.

\subsection{$2$-homotopy}\label{app.parte-hom}
 
We prove in detail that $2$-homotopy is an equivalence relation, and that it is compatible with the multiplication.

\begin{lemma} \label{lemma 2 homotopy equivalence relation}
	The $2$-homotopy of Definition~\ref{def.2-hom} defines an equivalence relation on $\Mor_{C^h_L(A)}(x,y) $.
\end{lemma}
\begin{proof}
 
	For any homotopy $  P(t,dt) \in \Mor_{C^h_L(A)}(x,y) $, we have  $  P(t,dt)\sim_{h}P(t,dt)$ via the 
	the $2$-homotopy  $R(t,s,dt,ds)=P(t,dt)$. If $  P(t,dt)=e^{p(t)}*x \sim_{h}Q(t,dt)=e^{q(t)}*x$ via  $R(t,s,dt,ds)=e^{r(t,s,ds)} *x \in \MC_{L[t,s,dt,ds]}(A) $ then $Q(t,dt)\sim_{h}P(t,dt)$ via   the $2$-homotopy $R'(t,s,dt,ds)=  e^{q(t) \bullet- r(t,s,ds)\bullet p(t) } *x\in \MC_{L[t,s,dt,ds]}(A) $.  In fact, the polynomial $q(t) \bullet -r(t,s,ds)\bullet p(t) \in  (L^0[t,s]\cdot t+ L^0[t,s]\cdot s+  L^{-1}[t,s]\cdot  tds) \otimes \m_A$ and 
	\[
	\begin{cases}
		R'(t,0,dt,0)= e^{q(t) \bullet - r(t,0,0)\bullet p(t) } *x=  e^{q(t)   } *x=Q(t,dt), \\
		R'(t,1,dt,0)= e^{q(t) \bullet- r(t,1,0)\bullet p(t) } *x =e^{p(t)}*x =P(t,dt), \\
		R'(0,s,0,ds)=  e^{q(0) \bullet -r(0,s,ds)\bullet p(0) } *x =x, \\
		R'(1,s,0,ds)= e^{q(1) \bullet -r(1,s,ds)\bullet p(1) } *x =e^{q(1) \bullet -r(1,s,ds)  } *y=e^{q(1)}*x=y.\\
	\end{cases}
	\]

	Finally, if $  P(t,dt)\sim_{h}Q(t,dt)= e^{q(t)}*x$,
	via  the $2$-homotopy $R(t,s,dt,ds)= e^{r(t,s,ds) } *x$, and    $  Q(t,dt)\sim_{h}S(t,dt)$, via the $2$-homotopy $R'(t,s,dt,ds)= e^{r'(t,s,ds) } *x$, then $  P(t,dt)\sim_{h}S(t,dt)$ via the $2$-homotopy 
	$R''(t,s,dt,ds)= e^{r'(t,s,ds) \bullet- q(t) \bullet r(t,s,ds) } *x$. In fact,  the polynomial $r'(t,s,ds) \bullet- q(t) \bullet r(t,s,ds) \in (L^0[t,s]\cdot t+ L^0[t,s]\cdot s+  L^{-1}[t,s]\cdot  tds) \otimes \m_A$ and
	\[
	\begin{cases}
		R''(t,0,dt,0)= e^{r'(t,0,0) \bullet -q(t) \bullet r(t,0,0) } *x= e^{q(t)\bullet -q(t) \bullet r(t,0,0) } *x=e^{r(t,0,0) } *x=P(t,dt), \\
		R''(t,1,dt,0)= e^{r'(t,1,1) \bullet -q(t) \bullet r(t,1,0) } *x=e^{r'(t,1,1) \bullet -q(t) \bullet q(t) } *x=e^{r'(t,1,1) } *x=S(t,dt), \\
		R''(0,s,0,ds)=  e^{r'(0,s,ds) \bullet- q(0) \bullet r(0,s,ds) } *x=  x, \\
		R''(1,s,0,ds)= e^{r'(1,s,ds) \bullet- q(1) \bullet r(1,s,ds) } *x =y,\\
	\end{cases}
	\]
	where we use  
	$R''(1,s,0,ds)= e^{r'(1,s,ds) \bullet- q(1) \bullet r(1,s,ds) } *x =e^{r'(1,s,ds) \bullet- q(1)} *y= e^{r'(1,s,ds)}*x=y $.
\end{proof}

\begin{lemma}\label{lemma 2 homotopy compatibile prodotto}
	The  $2$-homotopy equivalence relation is compatible with the multiplication given by
	\[
	\Mor_{C^h_L(A)}(y,z) \times  \Mor_{C^h_L(A)}(x,y) \to  \Mor_{C^h_L(A)}(x,z)
	\]
	\[(e^{q(t)}*y,e^{p(t)}*x) \mapsto  e^{q(t)\bullet p(t)}*x.
	\]

	\end{lemma}
\begin{proof}
	Suppose that $  P(t,dt)=e^{p(t)}*x\sim_{h}P'(t,dt)=e^{p'(t)}*x$  are $2$-homotopy equivalent via $R(t,s,dt,ds)=e^{r(t,s,ds) } *x \in \MC_{L[t,s,dt,ds]}(A) $  and $  Q(t,dt)=e^{q(t)}*y\sim_{h}Q'(t,dt)=e^{q'(t)}*y$ are $2$-homotopy equivalent via $R'(t,s,dt,ds)=e^{r'(t,s,ds) } *y \in \MC_{L[t,s,dt,ds]}(A) $. Then, the compositions  $e^{q(t)\bullet p(t)}*x \sim_h e^{q'(t)\bullet p'(t)}*x$ are $2$-homotopy equivalent  via
	$R''(t,s,dt,ds)=e^{r'(t,s,ds)\bullet r(t,s,ds) } *x \in \MC_{L[t,s,dt,ds]}(A) $.
	Indeed, we have 
	\begin{equation} 
		\begin{cases}
			R''(t,0,dt,0)= e^{r'(t,0,0)\bullet r(t,0,0) } *x= e^{q(t)\bullet p(t) } *x, \\
			R''(t,1,dt,0)=  e^{r'(t,1,0)\bullet r(t,1,0) } *x= e^{q'(t)\bullet p'(t) } *x,  \\
			R''(0,s,0,ds)= e^{r'(0,s,ds)\bullet r(0,s,ds) } *x=x,\\
			R''(1,s,0,ds)= e^{r'(1,s,ds)\bullet r(1,s,ds) } *x=z.\\
		\end{cases}
	\end{equation}
\end{proof}

\subsection{Locally free resolutions}\label{sezione appendice risoluzioni}
 In this part we outline some useful constructions with locally free resolutions.

\begin{lemma}\label{lemma.somma-diretta-risoluz}
 
	Let $f_1 \colon \EE_\FF^\cdot \to \FF$,  $f_2 \colon \EE_\FF'^\cdot \to \FF$, $g_1 \colon \EE_\GG^\cdot \to \GG$ and $g_2 \colon \EE_\GG'^\cdot \to \GG$ be bounded locally free resolutions of $\FF$ and $\GG$, respectively, such that $\EE_\FF^\cdot \hookrightarrow \EE_\GG^\cdot$  and $\EE_\FF'^\cdot \hookrightarrow \EE_\GG'^\cdot$ are  two bounded  locally free resolutions of $\FF \subset \GG$ that fit into a commutative diagram like \eqref{equazione risoluzione sottofascio}. Then, it is possible to construct a commutative diagram 
\begin{center}
		\begin{tikzcd}
			&                                                                  & \FF \arrow[ddd, dashed, hook, bend right]                        &                                             &   \\
			0 \arrow[r] & \EE_\FF^\cdot \oplus \EE_\FF'^\cdot \arrow[r] \arrow[d, hook] \arrow[ru,"f_1+f_2"] & \mathcal{Q}^\cdot \arrow[r] \arrow[d, hook] \arrow[u, two heads, "t"'] & \mathcal{R}^\cdot \arrow[r] \arrow[d, hook] & 0 \\
			0 \arrow[r] & \EE_\GG^\cdot\oplus \EE_\GG'^\cdot \arrow[r] \arrow[rd, "g_1+g_2"']                  & \mathcal{P}^\cdot \arrow[r] \arrow[d, two heads, "h"]                 & \mathcal{N}^\cdot \arrow[r]                 & 0 \\
			&                                                                  & \GG,                                                              &                                             &  
		\end{tikzcd}
	\end{center}
where $\mathcal{P}^\cdot,\mathcal{N}^\cdot, \mathcal{Q}^\cdot$ and $\mathcal{R}^\cdot$ are	bounded complexes of locally free sheaves,
	 the rows are short exact sequences, $t \colon \mathcal{Q}^\cdot \to \FF$ and $h \colon \mathcal{P}^\cdot \to \GG$ are  bounded locally free resolutions, and the induced morphisms $i_1 \colon \EE_\FF^\cdot \to \mathcal{Q}^\cdot$, $i_2 \colon \EE_\FF'^\cdot \to \mathcal{Q}^\cdot$, $j_1\colon\EE_\GG^\cdot \to \mathcal{P}^\cdot$, $j_2\colon\EE_\GG'^\cdot \to \mathcal{P}^\cdot$ are quasi-isomorphisms. 
\end{lemma}

\begin{proof}
	This proof follows \cite[Theorem 4.4]{joint-defos} and \cite[Lemma 7.7]{pairs}.
	Let $\mathcal{P}^k  = \mathcal{Q}^k=0$ for every  $k \geq 1$ 
	and $\mathcal{N}^k = \mathcal{R}^k =0$ for every $k \geq 0$. Let $\mathcal{Q}^0:= \EE_\FF^0 \oplus \EE_\FF'^0$,  $\mathcal{P}^0:= \EE_\GG^0 \oplus\EE_\GG'^0$, and define 
	$t= f_1+f_2 \colon \mathcal{Q}^0 \to \FF, h=g_1+g_2\colon \mathcal{P}^0 \to \GG$, with $f_1 \colon \EE_\FF^\cdot \to \FF$,  $f_2 \colon \EE_\FF'^\cdot \to \FF$, $g_1 \colon \EE_\GG^\cdot \to \GG$ and $g_2 \colon \EE_\GG'^\cdot \to \GG$ denoting the locally free resolutions of $\FF$ and $\GG$ respectively.
	
	We define $\mathcal{R}^{-1}$ and $\mathcal{N}^{-1}$ in such a way that $t$ and $h$ are isomorphisms on the zero cohomology level: the idea is to introduce, for  every new cycle in degree 0, a new generator in degree $-1$ whose differential  annihilates  the class. More explicitly, consider the fibre products $\EE_\FF^0 \times_{\FF} \EE_\FF'^0 \subset \EE_\GG^0 \times_{\GG} \EE_\GG'^0$ of the maps $f_1, f_2$ and $g_1, g_2$ respectively. Define 
	\[\mathcal{R}^{-1} = \bigoplus_{(x,x') \in\EE_\FF^0 \times_{\FF} \EE_\FF'^0} \OO_X \langle b_{(x,x')} \rangle,\]
	 \[\mathcal{Q}^{-1}= \EE_\FF^{-1}\oplus \EE_\FF'^{-1} \oplus  \mathcal{R}^{-1} = \EE_\FF^{-1}\oplus \EE_\FF'^{-1} \oplus \bigoplus_{(x,x') \in\EE_\FF^0 \times_{\FF} \EE_\FF'^0} \OO_X \langle b_{(x,x')} \rangle \] 
with $d_{\mathcal{Q}^\cdot}(e,e',b_{(x,x')})= (d_{\EE^\cdot_\FF}e +x, d_{\EE'^\cdot_\FF}e'-x')$, with $b_{(x,x')}$ being the generator of $\mathcal{R}^{-1}$ corresponding to $(x,x')\in \EE_\FF^0 \times_{\FF} \EE_\FF'^0$.  Define \[\mathcal{N}^{-1}= \bigoplus_{(y,y') \in\EE_\GG^0 \times_{\GG} \EE_\GG'^0} \OO_X \langle b_{(y,y')} \rangle,\] \[\mathcal{P}^{-1}= \EE_\GG^{-1}\oplus \EE_\GG'^{-1} \oplus  \mathcal{N}^{-1} = \EE_\GG^{-1}\oplus \EE_\GG'^{-1} \oplus \bigoplus_{(y,y') \in\EE_\GG^0 \times_{\GG} \EE_\GG'^0} \OO_X \langle b_{(y,y')} \rangle,\] 
with $d_{\mathcal{P}^\cdot}(h,h',b_{(y,y')})= (d_{\EE^\cdot_\GG} h +y, d_{\EE'^\cdot_\GG} h'-y')$, with $b_{(y,y')}$ being the generator of $\mathcal{N}^{-1}$ corresponding to $(y,y')\in \EE_\GG^0 \times_{\GG} \EE_\GG'^0$. Because of the inclusion $\EE_\FF^0 \times_{\FF} \EE_\FF'^0 \subset\EE_\GG^0 \times_{\GG} \EE_\GG'^0$, we can choose generators of $\EE_\FF^0 \times_{\FF} \EE_\FF'^0$ and $ \EE_\GG^0 \times_{\GG} \EE_\GG'^0$ in such a way that we obtain injective maps $\mathcal{R}^{-1} \hookrightarrow \mathcal{N}^{-1}$ and $\mathcal{Q}^{-1} \hookrightarrow \mathcal{P}^{-1}$.
	
	Now $t$ and $ h$ induce isomorphisms on the $0$th cohomologies, but we have introduced non-trivial cohomology in degree $-1$, so we define $\mathcal{R}^{-2}$ and $\mathcal{N}^{-2}$ to get rid of it. The $-1$-cocycles in $\mathcal{Q}^\cdot$ are the elements of the form $(e,e',b_{(x,x')})$ such that $d_{\EE^\cdot_\FF}e+x=0$, $d_{\EE'^\cdot_\FF}e'-x'=0$. If $b_{(x,x')}=0$, they are elements of the form $(e,e',0)$ such that $d_{\EE^\cdot_\FF}e=0$, $d_{\EE'^\cdot_\FF}e'=0$, i.e., $-1$-cocycles in $\EE_\FF^\cdot \oplus \EE_\FF'^\cdot$, and since the first negative cohomology vanishes, they are already coboundaries. For $b_{(x,y)} \neq 0$, we need to construct $\mathcal{R}^{-2}$ in such a way that all cocycles $(e,e',b_{(x,x')})$  are coboundaries. Consider the fibre product 
	\begin{center}
		\begin{tikzcd}
			I:={\EE_\FF^{-1}\oplus\EE_\FF'^{-1} \times_{\EE_\FF^{0}\oplus\EE_\FF'^{0}}\bigoplus_{(x,x')\in \EE_\FF^0 \times_\FF \EE_\FF'^0} \OO_X\langle b_{(x,x')} \rangle } \arrow[d] \arrow[r] & {\bigoplus_{(x,x')\in \EE_\FF^0 \times_\FF \EE_\FF'^0} \OO_X\langle b_{(x,x')} \rangle } \arrow[d, "\delta"] \\
			\EE_\FF^{-1}\oplus\EE_\FF'^{-1} \arrow[r, "-{d}"]                                                                                                           & \EE_\FF^0\oplus\EE_\FF'^0                                                                    
		\end{tikzcd}
	\end{center}
of the maps	$\delta(b_{(x,x')})= (x, -x')$ and $-{d}(e,e')=(-d_{\EE^\cdot_\FF}e, -d_{\EE'^\cdot_\FF}e')$.   Define 
\[ \mathcal{R}^{-2} = \!\!\! \!\!\! \!\!\! \bigoplus_{(e,e',b_{(x,x')}) \in I} \!\!\! \!\!\!\!\!\! \OO_X \langle c_{(e,e',b_{(x,x')})} \rangle, \  \mathcal{Q}^{-2}= \EE_\FF^{-2}\oplus \EE_\FF'^{-2} \oplus  \mathcal{R}^{-2} = \EE_\FF^{-2}\oplus \EE_\FF'^{-2} \oplus  \!\!\!\!\!\!  \!\!\! \bigoplus_{(e,e',b_{(x,x')}) \in I} \!\!\!\!\!\! \!\!\! \OO_X \langle c_{(e,e',b_{(x,x')})} \rangle,\]
with $d_{\mathcal{Q}^\cdot}(c_{(e,e',b_{(x,x')})}):= (e,e',b_{(x,x')})$.
Note that  $d_{\mathcal{Q}^\cdot}^2(z,w,c_{(e,e',b_{(x,x')})})= d_{\mathcal{Q}^\cdot}(d_{\EE^\cdot_\FF}z + e, d_{\EE'^\cdot_\FF}w+e', b_{(x,x')})= (d_{\EE^\cdot_\FF}e + x, d_{\EE'^\cdot_\FF}e' -x', 0)=0$.
 
Since there is an injective map
\[ I={\EE_\FF^{-1}\oplus\EE_\FF'^{-1} \times_{\EE_\FF^{0}\oplus\EE_\FF'^{0}}\!\!\!\!\!\!\!\!\!\bigoplus_{(x,x')\in \EE_\FF^0 \times_\FF \EE_\FF'^0} \!\!\!\!\!\!\!\!\!\OO_X\langle b_{(x,x')} \rangle }  \to J:={\EE_\GG^{-1}\oplus\EE_\GG'^{-1} \times_{\EE_\GG^{0}\oplus\EE_\GG'^{0}}\!\!\!\!\!\!\!\!\!\bigoplus_{(y,y')\in \EE_\GG^0 \times_\GG \EE_\GG'^0} \!\!\!\!\!\!\!\!\!\OO_X\langle b_{(y,y')} \rangle }, \]
as before we can choose generators of $I$ and $ J$ in such a way that we obtain injective maps $\mathcal{R}^{-2} \hookrightarrow \mathcal{N}^{-2}:= \bigoplus_{u \in J} \OO_X \langle c_u \rangle$, and $\mathcal{Q}^{-2} \hookrightarrow \mathcal{P}^{-2}:=  \EE_\GG^{-2}\oplus\EE_\GG'^{-2} \oplus  \mathcal{N}^{-2} $.

We have not introduced any extra cohomology in degree $-2$: the $-2$-cocycles in $\mathcal{Q}^\cdot$ are elements of the form
 $(z,w,c_{(e,e',b_{(x,x')})})$  such that $(d_{\EE^\cdot_\FF}z+e, d_{\EE'^\cdot_\FF}w+e', b_{(x,x')})=0$  which implies $b_{(x,x')}=0$, so by construction $c_{(e,e',b_{(x,x')})}=0$, so they are elements of the form  $(z,w,0)$ such that $d_{\EE^\cdot_\FF}z=d_{\EE'^\cdot_\FF}w=0$,
i.e., $-2$-cocycles in $\EE_\FF^\cdot \oplus \EE_\FF'^\cdot$, so they are coboundaries. The same holds for the $-2$-cocycles in $\mathcal{P}^\cdot$, so the process ends: $\mathcal{Q}^k = \EE_\FF^k \oplus \EE_\FF'^k$ and $\mathcal{P}^k = \EE_\GG^k \oplus \EE_\GG'^k$ for all $k \leq -3$.
\end{proof}

\begin{lemma}\label{lem.sollev-morf}
	Given a morphism of coherent sheaves $\alpha \colon \FF \to \GG$ there exist bounded locally free resolutions $\EE_\FF^{\cdot} \to \FF$  and  $\EE_\GG^{\cdot} \to \GG$ such that there exists a lifting of $\alpha \colon \FF \to \GG$ to  $\alpha^\cdot: \EE_\FF^\cdot \to \EE_\GG^\cdot$.
\end{lemma}

\begin{proof}
This is based on \cite[B.8.3]{Fulton}. Let  $\eta \colon \EE_\GG^{\cdot} \to \GG$ be any bounded locally free resolution of $\GG$: we will use this to construct a resolution of $\FF$ equipped with a lifting of $\alpha$. Define  $\EE^0_\FF$ to be a locally free sheaf such that there  exists a surjective map to the fibre product $\EE_\GG^0 \times_\GG \FF$ of the maps $\eta \colon \EE_\GG^0 \to \GG$ and $\alpha \colon \FF \to \GG$: 
	\begin{center}
		\begin{tikzcd}
			\EE_\FF^0 \arrow[rdd, "\alpha^0"', bend right=20] \arrow[rrd, bend left=20] \arrow[rd, two heads] &                                                    &                         \\
			&\EE_\GG^0\times_\GG \FF \arrow[d] \arrow[r, two heads] & \FF \arrow[d, "\alpha"] \\
			& \EE_\GG^0 \arrow[r, "\eta"', two heads]                & \GG .                   
		\end{tikzcd}               
	\end{center}
 The inductive step works similarly: assume we have built the complex $\EE_\FF^\cdot $ and the map $\alpha^\cdot$ up to degree $r \leq 0$, since $\EE_\GG^\cdot$ is  exact in negative degrees, $\im (d_{\EE_\GG^\cdot}^{r-1})= \ker (d_{\EE_\GG^\cdot}^r)$, so that we can choose $\EE_\FF^{r-1}$ to be a locally free sheaf such that there exists a surjective map to the fibre product $\EE_\GG^{r-1}\times_{\ker(d_{\EE_\GG^\cdot}^r)} \ker (d_{\EE^\cdot_\FF}^r)$, as in the diagram
	\begin{center}
		\begin{tikzcd}
			\EE_\FF^{r-1} \arrow[rdd, "\alpha^{r-1}"', bend right=20] \arrow[rrd, bend left=20] \arrow[rd, two heads] &                                                                                &                          \\
			& {\EE_\GG^{r-1}\times_{\ker (d_{\EE^\cdot_\GG}^r)} \ker (d_{\EE^\cdot_\FF}^r) }\arrow[d] \arrow[r, two heads] & \ker (d_{\EE^\cdot_\FF}^r) \arrow[d] \\
			& \EE_\GG^{r-1} \arrow[r, "d_{\EE^\cdot_\GG}^{r-1}"', two heads]                                     & \ker (d_{\EE^\cdot_\GG}^r)         . 
		\end{tikzcd}
	\end{center}
\end{proof}

\begin{corollary}\label{cor.grafico}
 		Given a morphism of coherent sheaves $\alpha \colon \FF \to \GG$ and a lifting  $\alpha^\cdot: \EE_\FF^\cdot \to \EE_\GG^\cdot$  of $\alpha \colon \FF \to \GG$ to bounded locally free resolutions as in Lemma~\ref{lem.sollev-morf}, the graph of $\alpha^\cdot$ is a bounded locally free resolution of the graph of $\alpha$.
\end{corollary}
\begin{proof}
	It is clear that $\gamma^\cdot := \operatorname{graph} (\alpha^\cdot)$ is a bounded complex of locally free sheaves, as it is isomorphic to $\EE_\FF^\cdot$. For the same reason, the complex $\gamma^\cdot$ is exact in all negative degrees, and the projection $\EE_\FF^\cdot \oplus \EE_\GG^\cdot \to \FF \oplus \GG$ induces a surjective map $\gamma^\cdot \to \gamma := \operatorname{graph}(\alpha)$.
\end{proof}

\end{document}